\documentclass[11pt]{article} 

\usepackage{amsmath}
\usepackage{amssymb}
\usepackage[T1]{fontenc}
\usepackage[latin1]{inputenc}
\usepackage{enumerate}

\def\r{\rightarrow}

\newcommand{\fdem}{\hspace*{\fill}~$\Box$\par\endtrivlist\unskip}

\newcommand{\E}{\mathbb{E}}     
\renewcommand{\P}{\mathbb{P}}     
\renewcommand{\L}{\mathbb{L}}

\newcommand{\N}{\mathbb{N}}     
\newcommand{\Z}{\mathbb{Z}}
\newcommand{\R}{\mathbb{R}}     
     
\newcommand{\C}{\mathbb{C}}

\newcommand{\diag}{\mathop{\rm diag}}

\renewcommand{\r}{\mathop{\rightarrow}}

\newcommand{\cB}{\mbox{$\cal B$}}
\newcommand{\cC}{\mbox{$\cal C$}}

\newcommand{\cE}{\mbox{$\cal E$}}
\newcommand{\cF}{\mbox{$\cal F$}}
\newcommand{\cG}{\mbox{$\cal G$}}
\newcommand{\cH}{\mbox{$\cal H$}}

\newcommand{\cR}{\mbox{$\cal R$}}

\newtheorem{theo}{Theorem}
\newtheorem{pro}{Proposition}

\newenvironment{proof}[1]{\textit{Proof#1.\,}}{\fdem}
\newtheorem{lem}{Lemma}
\newtheorem{rem}{Remark}
\newtheorem{cor}{Corollary}

\newtheorem{alem}{Lemma}[section]
\newtheorem{asublem}{Sublemma}[section]
\newtheorem{apro}{Proposition}[section]
\newtheorem{arem}{Remark}[section]

\title{Multidimensional renewal theory in the non-centered case. Application to strongly ergodic Markov chains. }
\author{D. GUIBOURG and L. HERV\'E \footnote{Universit\'e
Europ\'eenne de Bretagne, I.R.M.A.R. (UMR-CNRS 6625), Institut National des Sciences 
Appliqu\'ees de Rennes. Denis.Guibourg@ens.insa-rennes.fr, Loic.Herve@insa-rennes.fr} 
 }

\begin{document}
\maketitle



\vskip 2mm

\noindent{\bf Abstract.} {\scriptsize \it Let $(S_n)_{n\geq0}$ be a $\R^d$-valued random walk ($d\geq2$). Using Babillot's method \cite{bab}, we give general conditions on the characteristic function of $S_n$ under which $(S_n)_{n\geq0}$ satisfies the same renewal theorem as in the independent case (i.e.~the same conclusion as in the case when the increments of $(S_n)_{n\geq0}$ are assumed to be independent and identically distributed). 
This statement is applied to additive functionals of strongly ergodic Markov chains under the non-lattice condition and (almost) optimal moment conditions.}

\vskip 1mm

\section{Introduction}
Let $(S_n)_{n\geq0}$ be a $\R^d$-valued random walk. Renewal theory gives the behavior, as  $\|a\|\r+\infty$, of the positive measures  $U_a(\cdot)$ defined on the Borel $\sigma$-algebra $B(\R^d)$ of $\R^d$ as follows~: 
\begin{equation} \label{intro-Ua}
\forall a\in\R^d,\ \forall A\in B(\R^d),\quad U_a(A) = \sum_{n=1}^{+\infty}\E\big[1_A(S_n-a)\big]. 
\end{equation}
To define the renewal measure $U_a(\cdot)$, the sequence $(S_n)_{n\geq0}$ has to be transient: for independent or  Markov random walks, this leads to consider the following cases: \\[0.12cm]{\it 
1. $d\geq3$ and $\E[S_1]=0$ (centered case),  \\[0.12cm]
2. $d\geq1$ and $\E[S_1]\neq0$ (non-centered case): in this case, the behavior of $U_a(\cdot)$ is specified when $\|a\|\r+\infty$ in the direction of $\E[S_1]$. } \\[0.12cm]
The behavior of $U_a(\cdot)$ also depends on the usual lattice or non-lattice conditions. 

This work is the continuation of \cite{denis} (case $d=1$) and \cite{guiher} (centered case in dimension $d\geq3$). More specifically, in this paper, we consider the non-centered case in dimension $d\geq2$, and we present some general assumptions involving the characteristic function of $S_n$, 
under which we have the same conclusion as in the classical renewal theorem for random walks with independent and identically distributed (i.i.d.) increments. By using the weak spectral method \cite{fl}, this result is then applied to additive functionals of strongly ergodic Markov chains. This work is greatly inspired by Babillot's paper \cite{bab}. Before presenting our results, we give a brief review of well-known multidimensional renewal theorems in both independent and Markov settings, as well as some general comments on Fourier's method. To that effect we introduce some notations which will be repeatedly used afterwards:  

\noindent -$\ $ $\langle\cdot,\cdot\rangle$ is the canonical scalar product on $\R^d$, \\
-$\ $ $\|\cdot\|$ is the associated euclidean norm on $\R^d$, \\
-$\ $ $L_d(\cdot)$ is the Lebesgue-measure on $\R^d$,\\ 
-$\ $ $\cC_c(\R^d,\C)$ is the set of complex-valued continuous compactly supported functions on $\R^d$, \\
-$\ $ the Fourier transform of any Lebesgue-integrable function $f : \R^d\r\C$ is defined as follows: 
$\forall t\in\R^d,\ \hat{f}(t) := L_d( e^{-i\langle t, \cdot\rangle}f)$, \\
-$\ $ $\cH$ is the set of complex-valued continuous Lebesgue-integrable functions on $\R^d$, whose Fourier transform is compactly supported and infinitely differentiable on $\R^d$, \\
-$\ $ for any $R>0$, we denote by $B_R:=B(0,R)$ the open ball: 
$B(0,R) := \{t\in\R^d : \|t\|<R\}$, 
-$\ $ for any $0<r<b$, we denote by $K_{r,b}$ the annulus $K_{r,b} := \{t\in\R^d : r<\|t\|<b\}$. 

\noindent {\it Renewal theory for random walks with i.i.d.~non-centered increments.} \\[0.1cm]
Let $(X_n)_{n\geq1}$ be a sequence of i.i.d.~non-centered random variables (r.v) taking values in $\R^d$, and let $S_n = X_1+\ldots+X_n$. In dimension $d\geq2$, the renewal theorem was first established by Ney and Spitzer \cite{ney-spit} in the lattice case. Extension to the non-lattice case was obtained by Doney \cite{doney} under Cramer's condition, and by Stam \cite{stam} under the weaker non-lattice condition. Setting $m_d := \max(\frac{d-1}{2},2)$, Stam's statement writes as follows:   
\begin{equation} \label{intro-ren-decentre} 
\E\big[\|X_1\|^{m_d}\big] < \infty,\ \vec{m} := \E[X_1]\neq0\ \Rightarrow\ \forall g\in\cC_c(\R^d,\R),\  \lim_{\tau\r+\infty} \tau^{\frac{d-1}{2}}\, U_{\tau\vec m}(g) = C\, L_d(g)
\end{equation}
where $C$ is a positive constant depending on the first and second moments of $X_1$. In the lattice case, Property~(\ref{intro-ren-decentre}) still holds, but $L_d(\cdot)$ must be replaced with the product of counting  and Lebesgue measures both defined on some sublattices of $\R^d$. Stam's proof is based on the local limit theorem (LLT) due to Spitzer \cite[Th.~P7.10]{spitzer}\footnote{This LLT, established by Fourier techniques, extends to $d\geq 2$ the one-dimensional result of \cite{smith}.}.  More precisely, this LLT is applied  to study the difference  
$$\mbox{$\sum_{n=1}^{+\infty}n^{(d-1)/2}\big(\E[g(S_n-a)] - \E[g(T_n-a)]\big)$},$$ 
where the r.v.~$T_n$ are defined as the partial sums of a i.i.d.~sequence of Gaussian r.v.~having the same first and second moments as $X_1$. Then (\ref{intro-ren-decentre}) is deduced from the Gaussian case. 

\noindent {\it Fourier techniques in renewal theory (Breiman's method).} \\[0.1cm]
The weak convergence in (\ref{intro-ren-decentre}) can be established by investigating the behavior of $U_a(h)$ for $h\in\cH$. In fact the inverse Fourier formula gives (without any assumption on the model):  
\begin{equation} \label{intro-fle-fourier}
\mbox{$\forall h\in\cH,\quad \E\big[h(S_n-a)\big] = 
(2\pi)^{-d}\int_{\R^d}\hat h(t)\, \E[e^{i\langle t,S_n \rangle}]\, e^{-i\langle t,a \rangle}\, dt$}.
\end{equation}
This is the starting point of Fourier's method in probability theory. In the i.i.d.~case, using~(\ref{intro-fle-fourier}) and the formula $\E[e^{i\langle t,S_n \rangle}] = \E[e^{i\langle t,X_1 \rangle}]^n$, the potential $U_a(h)$ given by (\ref{intro-Ua}) is equal to the following integral: 
\begin{equation} \label{intro-Ua-fourier}
\mbox{$U_a(h) = I(a) := 
(2\pi)^{-d}\int_{K}\hat h(t)\, \frac{\phi(t)}{1-\phi(t)} \, e^{-i\langle t,a \rangle}\, dt\quad$ with $\quad \phi(t) = \E[e^{i\langle t,X_1 \rangle}]$},
\end{equation}
where $K$ is the support of $\hat h$. More precisely, since $\E[X_1]\neq0$ and $d\geq 2$, the integrand in $I(a)$ is integrable at $0$. Thus $I(a)$ is well-defined provided that $|\phi(t)| < 1$ for all $t\neq0$: this is the non-lattice condition. Fourier's method also applies to the lattice case by considering a periodic summation in (\ref{intro-Ua-fourier}). The renewal theorem then follows from the study of the integrals $I(\tau\vec m)$ when  $\tau\r+\infty$. This method, introduced by Breiman \cite{bre} in dimension $d=1$, was extended to $d\geq 2$ by Babillot \cite{bab} in the general setting of Markov random walks (see below). 

\noindent {\it Renewal theory for Markov random walks.} \\[0.1cm]
Let $(E,\cE)$ denote a measurable space, and let $(X_n,S_n)_{n\in\N}$ be an $E\times \R^d$-valued Markov random walk (MRW), namely: $(X_n,S_n)_{n\in\N}$ is a Markov chain and its transition kernel $P$ satisfies the following additive property (in the second component): 
\begin{equation} \label{def-mrw}
\forall(x,s)\in E\times\R^d,\ \forall A\in\cE,\ \forall B\in B(\R^d),\ \ \ P\big((x,s),A\times B\big) = 
P\big((x,0),A\times (B-s)\big).
\end{equation}
As usual we set $S_0=0$. When $(X_n)_{n\geq0}$ is strongly ergodic and $S_1$ is non-centered, Babillot gives in \cite{bab} some (operator-type) moment and non-lattice conditions for the additive component $(S_n)_n$ to satisfy the renewal conclusion in~(\ref{intro-ren-decentre}). Recall that the strong ergodicity condition states that the transition kernel $Q$ of $(X_n)_{n\geq0}$ admits an invariant probability measure $\pi$, and that there exists a Banach space $(\cB,\|\cdot\|_{\cal B})$, composed of $\pi$-integrable functions on $E$ and containing the function $1_E$, such that $\pi$ defines a continuous linear form on $\cB$ and 
\begin{equation} \label{K1}
\lim_{n\r+\infty} \sup_{f\in{\cal B},\, \|f\|\leq1} \|Q^nf - \pi(f)1_E\|_{\cal B} = 0.
\end{equation}
The proof in \cite{bab} is based on Fourier techniques and the usual Nagaev-Guivarc'h spectral method involving the semi-group of Fourier operators associated with $(X_n,S_n)_{n\in\N}$, namely: 
$$\forall n\in\N,\ \forall t\in\R^d,\ \forall x\in E,\  \ 
\big(Q_n(t)f\big)(x) := \E_{(x,0)}\big[e^{i \langle t , S_n \rangle} f(X_n)\big],$$ 
where $\E_{(x,0)}$ denotes the expectation under the initial distribution $(X_0,S_0)\sim\delta_{(x,0)}$. The operators  $Q_n(t)$ act (for instance) on the space of bounded measurable functions $f : E\rightarrow{\mathbb C}$. The semi-group property writes as follows: $\forall(m,n)\in\N^2,\ Q_{m+n}(t) = Q_m(t)\circ Q_n(t)$. In particular we have $Q_n(t) = Q_1(t)^n$. This property is the substitute for  MRWs of the formula $\E[e^{i\langle t,S_n \rangle}] = \E[e^{i\langle t,X_1 \rangle}]^n$ of the i.i.d.~case. 

\noindent {\it The content of the paper.} \\[0.1cm]
\indent Section~\ref{sect-met-fourier} focuses on Fourier's method. More specifically we consider a general sequence $(X_n,S_n)_{n\in\N}$ (not necessarily a MRW) of random variables taking values in $E\times\R^d$. In substance our  non-centered condition writes as follows: 
$\vec m := \lim_n\E[S_n]/n$ exists in $\R^d$ and is nonzero. Let $f : E\r [0,+\infty)$ be such that $\E[f(X_n)] < \infty$ for every $n\geq 1$. Under a general hypothesis, called $\cR(m)$, on the functions $t\mapsto \E\big[f(X_n)\, e^{i\langle t, S_n\rangle}\big]$, Theorem~\ref{ren-theo} states that there exists some positive constant $C$ (specified later) such that we have 
$$\forall g\in\cC_c(\R^d,\R),\quad \tau^{\frac{d-1}{2}}\sum_{n=1}^{+\infty}\E\big[f(X_n)\, g(S_n - a)\big]\longrightarrow C\, L_d(g)$$
when $a := a(\tau)\in\R^d$ goes to infinity "around the direction $\vec m$" in the sense (specified later) defined in \cite{thirion} (for instance, take $a:=\tau\vec m$ with $\tau\r +\infty$). Actually Hypothesis~$\cR(m)$, introduced in \cite{guiher} (centered case), contains the tailor-made conditions to prove renewal theorems via Fourier's method. The proof of Theorem~\ref{ren-theo} borrows the lines of \cite{bab} with the following improvements. First, the distribution-type arguments and the modified Bessel functions used in \cite{bab} are replaced  with elementary computations. Second, the (asymmetric) dyadic decomposition, partially developed in \cite{bab,bab2} to study integrals of type (\ref{intro-Ua-fourier}), is detailed in this work.

Section~\ref{sect-markov} is devoted to the Markov context. Specifically, we assume that $(X_n)_{n\in\N}$ is a Markov chain satisfying one of the three following classical strong ergodicity assumptions: \\[0.15cm] {\it 
- $\ (X_n)_{n\geq0}$ is $\rho$-mixing (see \cite{rosen}), \\
- $\ (X_n)_{n\geq0}$ is $V$-geometrically ergodic (see \cite{mey}), \\ 
- $\ (X_n)_{n\geq0}$ is a strictly contractive Lipschitz iterative model (see \cite{duf}).} \\[0.15cm] 
Let $\xi$ be a $\R^d$-valued measurable function, and let $S_n = \xi(X_1)+\ldots+\xi(X_n)$. Then 
the sequence $(X_n,S_n)_{n\in\N}$ is a special instance of MRW. As already used in \cite{guiher}, 
the weak spectral method \cite{fl} allows us to reduce Hypothesis~$\cR(m)$ to a non-lattice condition and to some (almost) optimal moment conditions on $\xi$, which are much weaker than those in \cite{bab}. 

Theorem~\ref{ren-theo} should supply further interesting applications, not only in Markov models but also in dynamical systems associated with quasi-compact Perron-Frobenius operators. On that subject, recall that the renewal theorems yield the asymptotic behavior of counting functions arising in the geometry of groups, as already developed for instance in \cite{lalley,broi,thirion}. 
\section{Renewal theory in the non-centered case (Fourier method)} \label{sect-met-fourier}
For any $A\subset\R^d$, $g : A\r\C$, and $\tau\in(0,1]$, we define the following quantities in $[0,+\infty]$: 
$$\big\|g\big\|_{0,A} = \sup_{x\in A}|g(x)|\ \ \ \mbox{and}\ \ \ 
\big[g\big]_{\tau,A}:=\sup\big\{\frac{|g(x)-g(y)|}{\|x-y\|^{\tau}},\ (x,y)\in A^2,\ x\not=y\big\}.$$
We say that $g$ is $\tau$-Hölder on $A$ if $[g]_{\tau,A}<\infty$. Moreover, for any open subset ${\cal O}$ of $\R^d$ and every $m\in\N^*$, we denote by $\cC_b^m({\cal O},\C)$ the vector space composed of $m$-times continuously differentiable functions $f : {\cal O}\r \C$ with bounded 
partial derivatives on ${\cal O}$. If $m\in(0,+\infty)\setminus\N$, we set $\tau := m-\lfloor m\rfloor$ where $\lfloor m\rfloor$ is the integer part of $m$, and we denote by $\cC_b^m({\cal O},\C)$ the vector space composed of functions $f : {\cal O}\r \C$ satisfying the three following conditions: { \it \\[0.12cm]
\indent $f$ is $\lfloor m\rfloor$-times continuously differentiable on ${\cal O}$, \\[0.12cm] 
\indent Each partial derivative of order $j=0,\ldots,\lfloor m\rfloor$ of $f$ is bounded on ${\cal O}$, \\[0.12cm]
\indent Each partial derivative of order $\lfloor m\rfloor$ of $f$ is $\tau$-hölder on ${\cal O}$. } 

\noindent Define $\nabla f := (\frac{\partial f}{\partial x_i})_{1\leq i\leq d}$ if $m\geq 1$, and  $Hess\, f := (\frac{\partial^2 f}{\partial x_i\partial x_j})_{1\leq i,j\leq d}$ (Hessian matrix) if $m\geq2$. 

Let $(\Omega,\cF,\P)$ be a probability space. We denote by $(E,\cE)$ a measurable space, and we consider a sequence $(X_n,S_n)_{n\geq0}$ of $E\times\R^d$-valued random variables defined on $\Omega$. Throughout, we assume $d\geq 2$. 
 
\noindent{\bf Hypothesis $\cR(m)$.}  {\it Given $m\in[2,+\infty)$ and $f : E\r[0,+\infty)$  a measurable function satisfying
\begin{equation} \label{cond-f-Rm}
\forall n\geq 1,\quad \E[f(X_n)] < \infty, 
\end{equation}
we say that Hypothesis $\cR(m)$ holds if the following conditions are fulfilled: }
\begin{enumerate}[(i)] {\it  
\item There exists $R>0$ such that, for all $t\in B_R$ and all $n\geq 1$, we have: 
\begin{equation} \label{X_n-S-n}
\E\big[f(X_n)\, e^{i\langle t, S_n\rangle}\big] = \lambda(t)^n\, L(t) + R_n(t), 
\end{equation}
where $\lambda(0)=1$, the functions $\lambda(\cdot)$ and $L(\cdot)$ are in $\cC_b^m\big(B_R,\C\big)$, and the series $\sum_{n\geq1}R_n(\cdot)$ uniformly converges on the open ball $B_R$ and defines a function in $ \cC_b^m\big(B_R,\C\big)$. 
\item  For all $0<r<b$, the series $\sum_{n\geq1}\E\big[f(X_n)\, e^{i\langle \cdot, S_n\rangle}\big]$ uniformly converges on the annulus $K_{r,b}$ and  defines a function in $ \cC_b^m\big(K_{r,b},\C\big)$.}
\end{enumerate}

Under Hypothesis $\cR(m)$, we set 
$$\vec m := -i\nabla\lambda(0) \quad \text{ and } \quad \Sigma := -Hess\, \lambda(0).$$
Below we assume that $\vec m \neq 0$: this is our non-centered condition. In fact, under Hypothesis~$\cR(m)$  and additional mild conditions (see \cite[Prop.~1]{guiher}), we have $\vec m = \lim_n\E[S_n]/n$, so that $\vec m$ may be viewed as a nonzero mean vector in $\R^d$. 
Below we also assume that the symmetric matrix $\Sigma$ is positive-definite. In the Markov setting of Section~\ref{sect-markov}, $\Sigma$ is linked to some covariance matrix, see (\ref{sigma}). 

\noindent{\bf  Hypothesis (H).} {\it Setting $m_d:=\max(\frac{d-1}{2},2)$, there exists a real number $m > m_d$ such that Hypothesis~$\cR(m)$ holds. We have $L(0)\not=0$, $\vec m \neq 0$, and $\Sigma := -Hess\, \lambda(0)$ is positive-definite.}
\begin{theo} \label{ren-theo} 
Assume that Hypothesis {\bf (H)} holds. Then, for each function $\mathfrak{a}\,:[0,+\infty)\to\R^d$ such that  
\begin{equation} \label{atau}
\mathfrak{A} := \lim_{\tau\to +\infty}\frac{\mathfrak{a}(\tau)-\tau \vec m}{\sqrt{\tau}}\ \ \mbox{exists in}\ \R^d,
\end{equation}
the family $\{V_{\tau}(\cdot),\, \tau\in(0,+\infty)\}$  of positive measures on $\R^d$ defined by 
\begin{equation} \label{serie-renou}
\forall A\in B(\R^d),\quad V_{\tau}(A) := (2\pi\tau)^{\frac{d-1}{2}}\sum_{n=1}^{+\infty}\E\big[f(X_n)\, 1_A(S_n - \mathfrak{a}(\tau))\big]
\end{equation}
weakly converges to $C\, L_d(\cdot)$ as $\tau \to +\infty$, where $C := C(L,\vec{m},\Sigma,\mathfrak{A})\in(0,+\infty)$ is given by: 
$$C(L,\vec{m},\Sigma,\mathfrak{A})\, :=\, L(0)\, \frac{\det S}{\|S\vec{m}\|}\, 
\exp\bigg(\frac{\big\langle S\vec{m}, S\mathfrak{A}\big\rangle^2 - \|S\vec{m}\|^2\|S\mathfrak{A}\|^2}{2\|S\vec{m}\|^2}\bigg)\quad \text{ with } \quad S:= \Sigma^{-\frac{1}{2}}.$$
\end{theo}

Condition~(\ref{atau}), introduced in \cite{thirion}, specifies what we called "around the direction $\vec m$" in Introduction. Obviously $\mathfrak{a}(\tau) = \tau \vec m$ satisfies (\ref{atau}).
\begin{rem} \label{rem-lattice}
Theorem~\ref{ren-theo} may be extended to the lattice case: Hypothesis $\cR(m)(ii)$ must be adapted, and $L_d(\cdot)$ is replaced with the product of the counting measure and the Lebesgue measure both associated with some sublattices of $\R^d$, see \cite[Sect.~2.5]{denis-these}.  
\end{rem}
The next subsections are devoted to the proof of  Theorem~\ref{ren-theo}. 
\subsection{Some reductions and Fourier techniques}
\noindent{\scriptsize $\bullet$} {\it Change of coordinates.} Let $\vec e_1$ denote the first vector of the canonical basis of $\R^d$, and let $T$ be any isometric linear map in $\R^d$ such that $T(\vec m)=\|\vec m\|\vec e_1$. Up to replace $S_n$ with $TS_n$, one may assume without loss of generality that $\vec{m}=\|\vec m\|\vec e_1$. 
This leads to replace $\lambda(\cdot),L(\cdot),R_n(\cdot),h(\cdot),\Sigma\ $ with 
$\ \lambda\circ T^{-1},L\circ T^{-1},R_n\circ T^{-1},h\circ T^{-1}, T\circ\Sigma\circ T^{-1}$. 

\noindent{\scriptsize $\bullet$} {\it The function $w$.} For all $x=(x_1,x_2,\ldots,x_d)\in\R^d$, we set $x' := (x_2,\ldots,x_d)\in\R^{d-1}$. The following function $w(\cdot)$ will play an important role: 
\begin{equation} \label{w} 
\forall x\in\R^d,\quad w(x)=-ix_1+\|x'\|^2.
\end{equation}
\begin{rem} \label{intdeunsurw} 
We have: $\forall x\in\R^d,\ |w(x)| \geq |x_1|^{3/4}\, \|x'\|^{1/2}$. Thus $1/w$ is integrable at 0. 
\end{rem}
\begin{rem} \label{fonction-v} 
Some simple facts on the function $\lambda(\cdot)$ of (\ref{X_n-S-n}) can be deduced from Hypothesis~{\bf (H)}. First, since $m_d\geq 2$, we have 
\begin{equation} \label{DL-lambda}
\mbox{$\lambda(t)=1+i\|\vec m\|t_1-\frac{1}{2}\langle \Sigma t,t\rangle + o(\|t\|^2)$}.
\end{equation}
Second, since $L(\cdot)$ is continuous on $B_R$ and $L(0)\not=0$, one may suppose (up to reduce $R$) that we have: $\forall t\in B_R,\ L(t)\not=0$. By Hypothesis~$\cR(m)(ii)$, the last property then implies that, for all $t\in B_R\setminus\{0\}$, the series  $\sum_{n\geq 1}\lambda(t)^n$ converges. Hence: 
$$\forall t\in B_R\setminus\{0\},\quad |\lambda(t)|<1.$$
Third the function 
$v_0(\cdot):=1-\lambda(\cdot)$ is in $ \cC_b^{m}\big(B_R,\C)$, and thanks to $-i\nabla\lambda(0)=\vec m=\|\vec m\|\vec e_1$, we have 
$$\forall j\in\{2,\ldots,d\},\quad \frac{\partial v_0}{\partial x_j}(0)=0.$$ 
Finally, up to reduce $R$, we deduce from (\ref{DL-lambda}) that there exist positive constants $\alpha$, $\beta$ 
such that: 
$$\forall t\in B_R,\quad \alpha|w(t)|\leq |v_0(t)|\leq \beta|w(t)|.$$
\end{rem}
\noindent{\scriptsize $\bullet$} {\it Use of the space $\cH$.} Thanks to the well-known results on weak convergence of positive measures (see \cite{bre}), Theorem~\ref{ren-theo} will hold provided that we prove the following property:
\begin{equation} \label{ren-h} 
\forall h\in\cH,\quad \lim_{\tau\to +\infty}\,(2\pi \tau)^{\frac{d-1}{2}}\sum_{n=1}^{+\infty}\E\big[\, f(X_n)\, h(S_n-\mathfrak{a}(\tau))\, \big] =
C(L,\vec{m},\Sigma,\mathfrak{A})\,\int_{\R^d} h(x)dx. 
\end{equation}
\noindent{\scriptsize $\bullet$} {\it Integral decomposition.}  Let $h\in\cH$ be fixed and let $b>0$ such that $\hat h(t)=0$ when $\|t\|>b$. Next consider any real number $\rho$ such that $0<\rho<\min(R,b)$ and any function $\chi\in \cC_b^\infty(\R^d,\R)$ compactly supported in $B_R$ such that $\chi(t)=1$ when $\|t\|\leq \rho$.  
For each $t\in\R^d$ we set $E_n(t) := \E[f(X_n)\, e^{i\langle t, S_n\rangle}]$. The inverse Fourier formula gives  
$$(2\pi)^{d}\, \E\big[f(X_n)\, h(S_n-a)\big] = 
\int_{\R^d}\hat h(t)\, E_n(t)\, e^{-i\langle t,a \rangle}\, dt.$$
Under Hypothesis {\bf (H)}, we prove below that, for every $a\in\R^d$, the following series  
$$I(a) := (2\pi)^d\, \sum_{n=1}^{+\infty}\E\big[f(X_n)\, h(S_n-a)\big]$$
converges, and that $I(a)$ decomposes as the sum of three integrals called $E(a)$, $E_1(a)$ and $I_1(a)$. The integrals $E(a)$ and $E_1(a)$ are error terms, while $I_1(a)$ is the main part of $I(a)$. In fact we have 
\begin{equation} \label{dec-int-fourier-E(a)}
I(a) = E(a) + \int_{B_R} \chi(t)\, \hat h(t)\, \frac{\lambda(t)}{1-\lambda(t)}\, L(t))\, e^{-i\langle t,a\rangle}\, dt
\end{equation}
with  
$$E(a) := \int_{\R^d} \hat h(t) \bigg( 1_{B_R}(t)\, \chi(t)\sum_{n\geq1}R_n(t) + 
1_{K_{\rho,b}}(t)\big(1-\chi(t)\big)\sum_{n\geq1}E_n(t)\bigg)\, e^{-i\langle t,a\rangle}\, dt.$$
Next we obtain 
\begin{equation} \label{dec-int-fourier}
I(a) = E(a) + E_1(a) + I_1(a) 
\end{equation}
with 
\begin{eqnarray}
E_1(a) &:=& 
\int_{\|t\| \leq R} \chi(t)\, \frac{\hat h(t)\lambda(t)\, L(t) - \hat h(0)\, L(0)}{1-\lambda(t)}\, 
e^{-i\langle t,a\rangle}\, dt \nonumber \\
&\ & 
\ +\ \ \hat h(0)\, L(0)\, \int_{\|t\| \leq R} \chi(t)\, \frac{\lambda(t) - 1-i\|\vec m\|t_1 + \frac{1}{2} \langle \Sigma t,t\rangle}
{(1-\lambda(t))(-i\|\vec m\|t_1+\frac{1}{2}\langle \Sigma t,t\rangle)}\, e^{-i\langle t,a\rangle}\, dt \label{E1}
\end{eqnarray}
and 
$$I_1(a) :=  
\, \hat h(0)\, L(0)\, \int_{\|t\| \leq R} \frac{\chi(t)}{-i\|\vec m\|t_1+\frac{1}{2}\langle \Sigma t,t\rangle}\, e^{-i\langle t,a\rangle}\, dt.$$

Such equalities are established in \cite[p.~389]{guiher} (centered case). By using Hypothesis {\bf (H)}, the proof of (\ref{dec-int-fourier-E(a)}) borrows the same lines. To obtain (\ref{dec-int-fourier}), use the fact that for $\|t\| \leq R$, $t\not=0$, we have $|\sum_{n=1}^N\lambda(t)^n| \leq 2/(b|w(t)|)$, and the fact that $1/w$ is integrable at 0 (cf.~Remarks~\ref{intdeunsurw}-\ref{fonction-v}).

Property~(\ref{ren-h}) then follows from (\ref{dec-int-fourier}) and the next properties (\ref{J(a)-K(a)}) (\ref{I1(a)-I3(a)}) and (\ref{I2(a)}).  
\subsection{Study of the first error term $E(a)$} \label{sub-Ja-Ka}
Here we prove that we have when $\|a\| \r+\infty$: 
\begin{equation} \label{J(a)-K(a)} 
E(a) = o(\|a\|^{-\frac{d-1}{2}}). 
\end{equation}
For $u\in \cC_b^ m(\R^d,\C)$ and $\alpha= (\alpha_1,\ldots,\alpha_d)\in\N^d$ such that $|\alpha|:=\sum_{i=1}^d\alpha_i \leq \lfloor m\rfloor$, we denote by  $\displaystyle{\partial^{\alpha}}$ the derivative operator defined by~:  
$$\partial^{\alpha} := \frac{\partial^{\vert\alpha\vert}}{\partial x_1^{\alpha_1}\ldots \partial x_d^{\alpha_d}} = \partial_1^{\alpha_1}\ldots\partial_d^{\alpha_d}\ \ \ \mbox{where}\ \ \ \partial_j := \frac{\partial}{\partial x_j}.$$
The following proposition is classical. Let ${\cal O}$ be a bounded open subset of $\R^d$. 
\begin{pro} \label{Cm-elementaire-infini} Let $m\in(0,+\infty)$ and $\tau=m-\lfloor m\rfloor$. Assume that $u$ is a function in $\cC_{b}^{\lfloor m\rfloor}(\R^d,\C)$ compactly supported in $\overline{{\cal O}}$ and that its restriction to ${\cal O}$ is in $\cC_b^m(\cal O,\C)$. Then the following properties hold:
\begin{enumerate}[(i)]
	\item $u\in\cC_b^m({\R^d},\C)$, and for $\alpha\in\N^d$, $|\alpha|=\lfloor m\rfloor$, we have:  $[\partial^{\alpha}u]_{\tau,\R^d}=[\partial^{\alpha}u]_{\tau,\overline{\cal O}} = 
[\partial^{\alpha}u]_{\tau,\cal O}$
\item $\exists C\in(0,+\infty),\ \forall a\in\R^d, \quad \|a\|^m|\hat{u}(a)|\leq C\big(\|u\|_{0,\R^d}+\sum_{|\alpha|=\lfloor m\rfloor}[\partial^{\alpha}u]_{\tau,\R^d}\big)$. 
\end{enumerate}
\end{pro}
\begin{proof}{ of (\ref{J(a)-K(a)})}
Define: 
$$\forall t\in B_R,\quad \cR(t) := \sum_{n=1}^{+\infty}R_n(t) \quad \text{and}\quad \forall t\in\R^d\setminus\{0\},\quad \cE(t) := \sum_{n=1}^{+\infty}E_n(t),$$
$$\forall t\in\R^d,\quad F(t) : =  1_{B_R}(t)\, \hat h(t)\, \chi(t)\, \cR(t) \quad \text {and}\quad 
G(t):= \left \{
    \begin{array}{ll}
      1_{K_{\rho,b}}(t)\, \hat h(t)\, (1-\chi(t))\, \cE(t) \quad \text{if } t\neq 0\\[0.12cm]
       0 \quad \text{if } t=0.
    \end{array}
\right. 
$$
We have 
$$\forall a\in\R^d,\quad E(a) = \hat{F}(a) + \hat{G}(a).$$ 
Note that $\chi\hat h\in\cC_c^{\infty}(\R^d,\C)$ is compactly supported in the closed ball $\overline{B}_R$, and that, by Hypothesis {\bf (H)}, we have $\cR\in \cC_b^{m}(B_R,\C)$ with $m > m_d$. By applying Proposition~\ref{Cm-elementaire-infini} to $u:=F$, we obtain $\|a\|^{m}|\hat{F}(a)| = O(1)$, thus 
$\hat{F}(a) = o(\|a\|^{-(d-1)/2})$ when $\|a\| \r+\infty$ since $m>(d-1)/2$. The same result holds for $\hat{G}(a)$ (replace $B_R$ with $K_{\rho,b}$). 
\end{proof}
\subsection{Study of the second error term $E_1(a)$} 
In this subsection we prove that we have when $\|a\|\to +\infty$: 
\begin{equation} \label{I1(a)-I3(a)} 
E_1(a) = o(\|a\|^{-(d-1)/2}). 
\end{equation}
Let $E_{11}(a)$ and $E_{12}(a)$ denote the two integrals in the right hand side of (\ref{E1}), so that we have: $E_1(a) := E_{11}(a) + E_{12}(a)$. Define: $\forall t\in B_R,\ \theta_1(t) = \chi(t)\big(\hat h(t)\lambda(t)\, L(t) - \hat h(0)\, L(0)\big)$. Then  
\begin{equation} \label{I1}
E_{11}(a) = \widehat{q_1}(a)\ \ \ \mbox{with}\ \ q_1 := 1_{B_R}\, \theta_1/v_0,
\end{equation} 
where  $v_0(t)=1-\lambda(t)$. Next define: $\forall t\in B_R,\ \theta_2(t) := \chi(t)\left(\lambda(t) - 1-i\|\vec{m}\|t_1 + \frac{1}{2} \langle \Sigma t,t\rangle\right)$ and $\tilde{v}_0(t) := -i\|\vec{m}\|t_1+\frac{1}{2}\langle \Sigma t,t\rangle$. 
Then   
\begin{equation} \label{I3}
E_{12}(a) = \widehat{q_2}(a)\ \ \ \mbox{with}\ \ q_2 := 1_{B_R}\, \theta_2/(v_0\, \tilde{v}_0).  
\end{equation} 

Unfortunately, since $q_1$ and $q_2$ are not defined at 0, $\widehat{q_1}(a)$ and $\widehat{q_2}(a)$ cannot be studied by the elementary arguments of  Subsection~\ref{sub-Ja-Ka}. This fact constitutes the main difficulty of the proof in  \cite{bab}. Below, we present the two key results  (Propositions~\ref{estimat-fourrier-q} and \ref{estimat-fourrier-q-bis}) to obtain (\ref{I1(a)-I3(a)}). In the next subsection, these two propositions are also used to obtain the desired result for the main part $I_1(a)$ (by difference with the Gaussian case, see Lemma~\ref{equivalent-particulier}). 

\noindent Recall that $w(\cdot)$ is defined in (\ref{w}). In the two next propositions, we consider any real numbers 
$m > m_d$ and $r>0$.  
\begin{pro} \label{estimat-fourrier-q}
Let $\theta$ and $v$ be complex-valued functions on $B_{r}$ such that: \\[0.12cm]
\indent {\scriptsize $\bullet$} $\theta\in\cC_b^m\big(B_{r},\C\big)$ with compact support in $B_r$ and $\theta(0)=0$, \\
\indent {\scriptsize $\bullet$} $v\in\cC_b^m\big(B_{r},\C\big)$ and  
\begin{equation} \label{(i)}
\forall j\in\{2,\ldots,d\}, (\partial_j v)(0)=0
\end{equation}  
\begin{equation} \label{(ii)}
\exists (a,b)\in(\R_+^*)^2,\ \forall x\in B_{r},\ \ a|w(x)|\leq |v(x)|\leq b|w(x)|.
\end{equation}

Then $q := 1_{B_{r}}\, \theta/v$ is integrable on $\R^d$ and $\lim_{\|a\|\to+\infty}\|a\|^{(d-1)/2}\, \hat{q}(a)=0$. 
\end{pro}
\begin{pro} \label{estimat-fourrier-q-bis}
In addition to the hypotheses of Proposition~\ref{estimat-fourrier-q}, we consider another complex-valued function $\tilde{v}$ on $B_{r}$ satisfying the same hypotheses as $v(\cdot)$. Moreover we assume that all the first and second partial derivatives of $\theta$ vanish at 0. Then $q := 1_{B_{r}}\, \theta/(v\,\tilde{v})$ is integrable on $\R^d$ and $\lim_{\|a\|\to +\infty}\|a\|^{(d-1)/2}\hat{q}(a)=0$.
\end{pro}
The proofs of Propositions~\ref{estimat-fourrier-q}-\ref{estimat-fourrier-q-bis}, based on dyadic decompositions, are partially presented in \cite{bab2,bab}. Since dyadic decomposition is not familiar to probabilistic readers, 
these proofs are detailed in Appendix~\ref{B} for the sake of completeness. 

\noindent \begin{proof}{ of (\ref{I1(a)-I3(a)})} 
Proposition~\ref{estimat-fourrier-q} applied with $\theta:=\theta_1$ and $v:=v_0$ (see Rk.~\ref{fonction-v}) gives 
$\widehat{q_1}(a) = o(\|a\|^{-(d-1)/2})$ when $\|a\|\to+\infty$. Similarly Proposition~\ref{estimat-fourrier-q-bis} applied with $\theta:=\theta_2$,  $v:=v_0$ and $\tilde{v}:=\tilde{v}_0$ gives $\widehat{q_2}(a) = o(\|a\|^{-(d-1)/2})$. Then (\ref{I1(a)-I3(a)}) follows from (\ref{I1}) (\ref{I3}).  
\end{proof}

\subsection{Study of the main part $I_1(a)$ of $I(a)$} \label{sub-I(a)}
Let $\mathfrak{a} : \R^+\r \R^d$ be a measurable function satisfying (\ref{atau}). Here we prove that  
\begin{equation} \label{I2(a)}
\mbox{$\lim_{\tau\to+\infty} \tau^{(d-1)/2}I_1(\mathfrak{a}(\tau))=(2\pi)^{(d+1)/2}\, C(L,\vec{m},\Sigma,\mathfrak{A})\,\, \hat{h}(0)$}. 
\end{equation} 
The proof of (\ref{I2(a)}) in \cite{bab} involves the modified Bessel functions and some related computations partially made in the book \cite{wat}.   
Here we present a direct and simpler proof of (\ref{I2(a)}) based on the next proposition. We denote by $S(\R^d)$ the so-called Schwartz space. 
\begin{pro}\label{equivalent}
Let $\vec{w}\in\R^d\setminus\{0\}$,  and let $\mathfrak{p}\,:\,\R^+\to \R^{d}$ such that $\mathfrak{P} := \lim_{\tau\to +\infty}\mathfrak{p}(\tau)/\sqrt{\tau}$ exists in $\R^{d}$. Then we have for all function $k\in S(\R^d)$ 
$$
\lim_{\tau\to+\infty}\tau^{\frac{d-1}{2}}\int_{\R^d}\frac{k(u)\,e^{-i\langle u,\tau\vec{w}+\mathfrak{p}(\tau)\rangle}}{-i\langle \vec{w},u\rangle+\|u\|^2}\,du 
= \frac{2\,\pi^{\frac{d+1}{2}}}{\|\vec{w}\|}\,k(0) \\ 
\exp\bigg(-\frac{\|\mathfrak{P}\|^2\|\vec{w}\|^2-\langle\mathfrak{P},\vec{w}\rangle^2}{4\|\vec{w}\|^2}\bigg).
$$
\end{pro}
The proof of Proposition~\ref{equivalent} (again based on Propositions~\ref{estimat-fourrier-q}-\ref{estimat-fourrier-q-bis}) is presented below. Let us first apply Proposition~\ref{equivalent} to establish (\ref{I2(a)}). 

\noindent \begin{proof}{ of (\ref{I2(a)})} Since $\vec{m}=\|\vec{m}\|\vec{e_1}$, one can rewrite (\ref{atau}) as follows: $\mathfrak{a}(\tau)=\tau\,\|\vec{m}\|\vec{e_1}+\sqrt{\tau}\, \mathfrak{b}(\tau)$  with $\mathfrak{b} :[0,+\infty)\to\R^d$ such that $\lim_{\tau\to +\infty}\mathfrak{b}(\tau)=\mathfrak{A}$. Denote by $\lambda_1,\ldots,\lambda_d$ the (positive) eigenvalues of  $\Sigma$. Let 
$\Delta := \diag(\sqrt{\lambda_1},\ldots,\sqrt{\lambda_d})$ and let $P$ be any orthogonal $d\times d$-matrix such that $$P^{-1}\Sigma P = \Delta^2 = \diag(\lambda_1,\ldots,\lambda_d).$$
Observe that $\langle \Sigma t,t\rangle = \langle \Delta^2 P^{-1}t ,  P^{-1}t\rangle = \|\Delta  P^{-1}t\|^2$. Set  $\vec{\ell} := \Delta^{-1}\,P^{-1}\,\vec{e_1}$. By using the variable $t=P\Delta^{-1}u$, one obtains (use $t_1=\langle P\Delta^{-1}u,\vec{e_1}\rangle=\langle u,\vec{\ell}\rangle$) 
$$I_1(a) = 2\hat{h}(0)\, L(0)\,(\det\Delta)^{-1} \int \, \frac{\chi(P\Delta^{-1}u)e^{-i\langle \Delta^{-1}u,P^{-1}\,a\rangle}}{-i\langle 2\|\vec{m}\|\vec{\ell},u\rangle +\|u\|^2} \, du.$$
Set $\zeta(x) := \chi(P\Delta^{-1} x)\, $ ($x\in\R^d$) and $\mathfrak{p}(\tau) := \sqrt{2\tau} \Delta^{-1} P^{-1} \mathfrak{b}(2\tau)\, $ ($\tau>0$). From the equality $\langle \Delta^{-1}u,P^{-1}\vec{e_1}\rangle = \langle u,\vec{\ell}\rangle$, it follows that 
\begin{eqnarray*}
I_1(\mathfrak{a}(\tau))&=&2\hat{h}(0)\, L(0)\,(\det\Delta)^{-1}\int \,\frac{\chi(P\Delta^{-1}u)e^{-i\big\langle \Delta^{-1}u\, ,\, P^{-1}\big(\tau\,\|\vec{m}\|\vec{e_1}+\sqrt{\tau}\,\mathfrak{b}(\tau)\big)\big\rangle}}{-i\langle 2\|\vec{m}\|\vec{\ell},u\rangle +\|u\|^2} \, du\\
&=& 2\hat{h}(0)\, L(0)\,(\det\Delta)^{-1}\,\int \,\frac{\zeta(u)e^{-i\big\langle u\, ,\, 2(\frac{\tau}{2})\|\vec{m}\|\vec{\ell}+\mathfrak{p}(\frac{\tau}{2})\big\rangle}}{-i\langle 2\|\vec{m}\|\vec{\ell},u\rangle +\|u\|^2} \, du. 
\end{eqnarray*} 
Now $\Delta^{-1}=P^{-1}\Sigma^{-\frac{1}{2}}P$ gives $\vec{\ell}=P^{-1}\Sigma^{-\frac{1}{2}}\vec{e_1}$, so  $\|\vec{\ell}\|=\|\Sigma^{-\frac{1}{2}}\vec{e_1}\|$ and $\|\Sigma^{-\frac{1}{2}}\vec{m}\|=\|\vec{m}\|\|\vec{\ell}\|$. Moreover, we have  $\zeta(0)=\chi(0)=1$ and 
$$\lim_{\tau\to+\infty} \mathfrak{p}(\tau)/\sqrt{\tau} = \mathfrak{P} \quad \text{with}\quad \mathfrak{P} := \sqrt{2}\,\Delta^{-1}P^{-1}\, \mathfrak{A}=\sqrt{2}\,P^{-1}\Sigma^{-\frac{1}{2}}\, \mathfrak{A}.$$ 
From Proposition~\ref{equivalent}, applied with $\mathfrak{P}$ previously defined, $\vec{w}:= 2\|\vec{m}\|\vec{\ell}=2P^{-1}\Sigma^{-\frac{1}{2}}\vec{m}$, and finally with the function $k := \zeta$, one obtains:
\begin{eqnarray*}
\lim_{\tau\to +\infty}(\frac{\tau}{2})^{\frac{d-1}{2}}I_1(\mathfrak{a}(\tau)) 
&=& 2\hat{h}(0)\, L(0)\,(\det\Sigma)^{-\frac{1}{2}}\,\frac{2\pi^{\frac{d+1}{2}}}{2\|\Sigma^{-\frac{1}{2}}\vec{m}\|} \\ 
&\ & \quad \quad \quad \quad \times\quad \exp\bigg(-\frac{\|\Sigma^{-\frac{1}{2}}\vec{m}\|^2\|\Sigma^{-\frac{1}{2}}\mathfrak{A}\|^2 -\langle\Sigma^{-\frac{1}{2}}\vec{m}, \Sigma^{-\frac{1}{2}}\mathfrak{A}\rangle^2}{2\|\Sigma^{-\frac{1}{2}}\vec{m}\|^2}\bigg),
\end{eqnarray*}
from which we easily deduce (\ref{I2(a)}). 
\end{proof}
\begin{proof}{ of Proposition \ref{equivalent}}
Let $U$ be an isometric linear map on $\R^d$ such that $U(\vec{w})=\|\vec{w}\|\vec{e_1}$. Let $\tau>0$. The change of  variable $v=(v_1,v')=U(u)$ in the integral of Proposition \ref{equivalent} gives 
$$\displaystyle\int_{\R^d}\frac{k(u)\,e^{-i\langle u,\tau \vec{w}+\mathfrak{p}(\tau)\rangle}}{-i\langle \vec{w},u\rangle+\|u\|^2}\,du=
\int_{\R^d}\frac{k(U^{-1}(v))\,e^{-i\tau\|\vec{w}\|v_1}e^{-i\langle v,U(\mathfrak{p}(\tau))\rangle}}{-i\|\vec{w}\|v_1+\|v\|^2}\,dv,$$
and by hypothesis we know that $\lim_{\tau\to +\infty}U(\mathfrak{p}(\tau))/\sqrt{\tau}=U(\mathfrak{P})$. 
Set $U(\mathfrak{P}):=(\ell_1,\ell')$ with $\ell_1\in\R$ et $\ell'\in\R^{d-1}$. As $U$ is isometric, we have $\|U(\mathfrak{P})\|=\|\mathfrak{P}\|$ and  $\ell_1=\langle U(\mathfrak{P}),\vec{e_1}\rangle=\langle \mathfrak{P},U^{-1}(\vec{e_1})\rangle = \langle \mathfrak{P},\vec{w}\rangle/\|\vec{w}\|$. Thus 
$$\|\ell'\|^2=\|\mathfrak{P}\|^2 - \langle \mathfrak{P},\vec{w}\rangle^2/\|\vec{w}\|^2 = \big(\|\vec{w}\|^2\|\mathfrak{P}\|^2-\langle \mathfrak{P},\vec{w}\rangle^2\big)/\|\vec{w}\|^2.$$
Next, let us set $U(\mathfrak{p}(\tau)):=\mathfrak{u}(\tau)=(\mathfrak{u}_{1}(\tau),\cdots,\mathfrak{u}_d(\tau))$, and $\mathfrak{u}'(\tau)=(\mathfrak{u}_2(\tau),\cdots,\mathfrak{u}_d(\tau))$. Then 
$$\displaystyle\int_{\R^d}\frac{k(u)e^{-i\langle u,\tau \vec{w}+\mathfrak{p}(\tau)\rangle}}{-i\langle \vec{w},u\rangle+\|u\|^2}\,du=
\int_{\R^d}\frac{k(U^{-1}(v))\,e^{-i(\tau\|\vec{w}\|+\mathfrak{u}_1(\tau))v_1}e^{-i\langle \mathfrak{u}'(\tau),v'\rangle}}{-i\|\vec{w}\|v_1+\|v\|^2}\,dv.$$
Observe that the function $k\circ U^{-1}$ is in $S(\R^d)$ and that 
\begin{equation} \label{u-v-tau}
\lim_{\tau\to +\infty} \mathfrak{u}_1(\tau)/\tau = 0 \quad \text{and} \quad \lim_{\tau\to +\infty} \mathfrak{u}'(\tau)/\sqrt{\tau} = \ell'.
\end{equation} 
Consequently Proposition~\ref{equivalent} will be established if we prove that, for any $\mu\in\R$, $\mu\neq0$, and any $h\in S(\R^d)$, we have:  
\begin{equation} \label{equivalent-bis}
\lim_{\tau\to+\infty}\tau^{(d-1)/2}\int_{\R^d}\frac{h(v)\,e^{-i(\tau \mu+\mathfrak{u}_1(\tau)) v_1}e^{-i\langle \mathfrak{u}'(\tau),v'\rangle}}{-i\mu v_1+\|v\|^2}\,dv\, = \frac{2\pi^{(d+1)/2}\, e^{- \|\ell\, '\|^2/4}}{|\mu|}\,h(0).
\end{equation}
\begin{lem} \label{equivalent-particulier} 
Property~(\ref{equivalent-bis}) holds with $h(v)=H(v) := e^{-(\mu^2 v_1^2 + \|v'\|^4)/2}$, namely: 
\begin{equation} \label{equivalent-cas-part}
J_{\mu}(\tau):=\int_{\R^d}\frac{e^{-(\mu^2 v_1^2 + \|v'\|^4)/2}\,e^{-i(\tau \mu+\mathfrak{u}_1(\tau)) v_1}e^{-i\langle \mathfrak{u}'(\tau),v'\rangle}}{-i\mu v_1+\|v'\|^2}\,dv\, \sim_{\tau\to +\infty}  \frac{2\pi^{(d+1)/2}
e^{-\|\ell\, '\|^2/4}}{|\mu|\, \tau^{(d-1)/2}}.
\end{equation}
\end{lem}  

Let us first assume that Lemma~\ref{equivalent-particulier} is valid, and let us deduce (\ref{equivalent-bis}) from (\ref{equivalent-cas-part}). To that effect, we shall proceed by difference and use Propositions~\ref{estimat-fourrier-q} and \ref{estimat-fourrier-q-bis}. For any $\tau>0$, we define 
$\mathfrak{d}(\tau) := (\tau\mu+\mathfrak{u}_1(\tau),\mathfrak{u}'(\tau))$. Moreover we set 
\begin{equation} \label{Delta}
\forall v\in\R^d\setminus\{0\}, \quad 
\Delta(v):=\frac{h(v)-h(0)H(v)}{-i\mu v_1+\|v\|^2}-h(0)\frac{v_1^2H(v)}{(-i\mu v_1+\|v\|^2)(-i\mu v_1+\|v'\|^2)}. 
\end{equation}
We have:   
\begin{eqnarray*}
\int_{\R^d}\frac{h(v)\,e^{-i(\tau \mu+\mathfrak{u}_1(\tau)) v_1}e^{-i\langle \mathfrak{u}'(\tau),v'\rangle}}{-i\mu v_1+\|v\|^2}\,dv-h(0)J_{\mu}(\tau)
&=&\int_{\R^d}e^{-i(\tau \mu+\mathfrak{u}_1(\tau)) v_1}e^{-i\langle \mathfrak{u}'(\tau),v'\rangle}\,\Delta(v)\, dv \\
&=& \int_{\R^d}e^{-i\langle \mathfrak{d}(\tau),\, v\rangle}\,\Delta(v)\, dv. 
\end{eqnarray*}
Thanks to (\ref{equivalent-cas-part}), Property~(\ref{equivalent-bis}) will be established provided that we prove the following:  
\begin{equation} \label{enonce}
\lim_{\tau\to +\infty} \tau^{\frac{d-1}{2}} \int_{\R^d}e^{-i\langle \mathfrak{d}(\tau),\, v\rangle}\,\Delta(v)\, dv=0.\\
\end{equation} 
Let us consider the functions ${G}_1$ and ${G}_2$ in $S(\R^d)$ defined by ${G}_1(v)=h(v)-h(0)H(v)$ and ${G}_2(v)=v_1^2H(v)$. One can easily check that ${G}_1(0)={G}_2(0)=0$, and 
\begin{subequations}
\begin{eqnarray*}
& & \forall j\in\{1,\ldots,d\},\quad (\partial_j {G}_2)(0)=0\quad \text{and}\quad (\partial_1^2 {G}_2)(0)=2, \\
& & \forall (j,\ell)\in\{1,\ldots,d\}^2\setminus\{(1,1)\},\quad (\partial^2_{j\ell}\,  {G}_2)(0)=0. 
\end{eqnarray*}
\end{subequations}
with $\partial^2_{j\ell} := \frac{\partial^2}{\partial x_j\partial x_\ell}$.
Next, let $g\in S(\R^d)$ such that $g(0)=- 1/\mu^2$, and define 
$$\forall v=(v_1,\ldots,v_d)\in\R^d,\quad s(v) := \big(-i\mu v_1+\|v\|^2\big)\, \big(-i\mu v_1+\|v'\|^2\big)\, g(v),$$
where $v'=(v_2,\ldots,v_d)$. One can easily check that $s(0)=0$, $\forall j\in\{1,\ldots,d\},\ (\partial_js)(0)=0$, $(\partial_1^2s)(0)=-2\mu^2g(0)=2$ and 
$\forall (j,l)\in\{1,\ldots,d\}^2\setminus\{(1,1)\},\ (\partial^2_{j\, l}s)(0)=0$. 
Therefore the first and second derivatives of the difference $ {G}_3:= {G}_2-s$ vanish at $0$. Rewrite $\Delta(v)$ as: 
$$\forall v\in\R^d\setminus\{0\}, \quad \Delta(v) = \frac{ {G}_1(v)}{-i\mu v_1+\|v\|^2}\ -\ h(0)\frac{ {G}_3(v)}{(-i\mu v_1+\|v\|^2)(-i\mu v_1+\|v'\|^2)}\ -\ h(0)g(v).$$
Let $\gamma\in \cC_b^{\infty}(\R^d,[0,1])$ be compactly supported and such that $\gamma_{|B}=1$ for some closed ball $B$ of $\R^d$ centered at 0. Since $\widehat{\gamma g}\in S(\R^d)$ and 
$\lim_{\tau\to +\infty}\|\mathfrak{d}(\tau)\|=+\infty$, we have  
$$\lim_{\tau\to +\infty} \|\mathfrak{d}(\tau)\|^{\frac{d-1}{2}}\int_{\R^d} 
e^{-i\langle \mathfrak{d}(\tau),\, v\rangle}\,\gamma(v)\, g(v)\, dv=0.$$
Moreover Propositions~\ref{estimat-fourrier-q} and \ref{estimat-fourrier-q-bis} yield the following properties: 
\begin{subequations}
\begin{eqnarray*}
& & \lim_{\tau\to +\infty} \|\mathfrak{d}(\tau)\|^{(d-1)/2}\int_{\R^d} 
e^{-i\langle \mathfrak{d}(\tau),\, v\rangle}\,\frac{\gamma(v) {G}_1(v)}{{-i\mu v_1+\|v\|^2}}\, dv=0, \\
& & \lim_{\tau\to +\infty} \|\mathfrak{d}(\tau)\|^{(d-1)/2}\int_{\R^d} 
e^{-i\langle \mathfrak{d}(\tau),\, v\rangle}\,\frac{\gamma(v) {G}_3(v)}{(-i\mu v_1+\|v\|^2)(-i\mu v_1+\|v'\|^2)}\, dv=0.
\end{eqnarray*}
\end{subequations}
Note that we have $\|\mathfrak{d}(\tau)\|\sim_{\tau\to+\infty} |\mu|\tau$ from $\mathfrak{d}(\tau) = (\tau\mu+\mathfrak{u}_1(\tau),\mathfrak{u}'(\tau))$ and (\ref{u-v-tau}). Hence:   
\begin{equation} \label{gammadelta}
\lim_{\tau\to +\infty} \tau^{(d-1)/2}\int_{\R^d}e^{-i\langle v,\mathfrak{d}(\tau)\rangle}\,\gamma(v)\, \Delta(v)\, dv=0. 
\end{equation}
Now observe that all the derivatives of $\gamma(\cdot)$ are bounded on $\R^d$ and that $\Delta$ is defined on $\R^d\setminus\{0\}$ as the quotient of a function of $S(\R^d)$ by a polynomial function (cf.~(\ref{Delta})). 
Since $1-\gamma(\cdot)$ vanishes on the closed ball $B$ (centered at 0), we deduce that $(1-\gamma)\Delta\in S(\R^d)$. So the Fourier transform of $(1-\gamma)\Delta$ is in $S(\R^d)$ and we have  
\begin{equation} \label{gammadelta2}
\lim_{\tau\to +\infty} \tau^{(d-1)/2}\int_{\R^d} 
e^{-i\langle \mathfrak{d}(\tau),\, v\rangle}\,(1-\gamma(v))\Delta(v)dv=0.
\end{equation}
Hence (\ref{enonce}) follows from (\ref{gammadelta}) and (\ref{gammadelta2}). As already mentioned, we have    
(\ref{enonce}) $\Rightarrow$ (\ref{equivalent-bis}), so that Proposition~\ref{equivalent} is proved.  
\end{proof}  
\noindent\begin{proof}{ of Lemma~\ref{equivalent-particulier}} It suffices to prove Lemma~\ref{equivalent-particulier} in case $\mu=1$ (if not, set $w_1=\mu v_1,\,w'=v'$). Since $\int_0^{+\infty} v_1^{-3/4}\, e^{- v_1^2/2} \,dv_1<\infty$ and $\int_{\R^{d-1}} \|v'\|^{- 1/2}\,  e^{-\|v'\|^4/2} \,dv'<\infty$ (use $\frac{1}{2}<d-1$ for the second integral), it follows from the Fubini-Tonelli theorem that $\int_{\R^d} \frac{e^{-(v_1^2 + \|v'\|^4)/2}}{|-iv_1+\|v'\|^2|}\,dv < \infty$ (cf.~Rk.~\ref{intdeunsurw}). Next we have:  
\begin{subequations}
\begin{eqnarray}
& & \mbox{$\forall n\in\N^*,\ \forall b\in\R^n,\quad (2\pi)^{-n/2}\int_{\R^n}e^{-\|x\|^2/2}\, 
e^{-i\langle x,b\rangle}dx=e^{- \|b\|^2/2}$} \label {Gauss} \\
& &  \mbox{$\forall v'\not=0,\quad \int_0^{+\infty}e^{(iv_1-\|v'\|^2)u}du=\frac{1}{-iv_1+\|v'\|^2}$}. \label{départ}
\end{eqnarray}
\end{subequations}
From (\ref{départ}), Fubini's theorem and $(\ref{Gauss})$, it follows that 
\begin{eqnarray*}J_{1}(\tau)&=&\int_{\R^{d-1}}e^{-i\langle \mathfrak{u}'(\tau),v'\rangle}e^{-\|v'\|^4/2}\left(\int_{\R}\frac{e^{-i(\tau +\mathfrak{u}_1(\tau)) v_1}e^{-v_1^2/2}}{-iv_1+\|v'\|^2}dv_1 \right)dv' \\   
&=&\int_{\R^{d-1}}e^{-i\langle\mathfrak{u}'(\tau),v'\rangle}e^{-\|v'\|^4/2}\left(\int_0^{+\infty}e^{-\|v'\|^2u}\,\big[\int_{\R}e^{-i(\tau +\mathfrak{u}_1(\tau) -u)v_1}e^{-v_1^2/2}dv_1\big]\,du\right)\,dv' \\
&=&\sqrt{2\pi}\int_{\R^{d-1}}e^{-i\langle \mathfrak{u}'(\tau),v'\rangle}e^{-\|v'\|^4/2}\left(\int_0^{+\infty}e^{-\|v'\|^2u}e^{-(\tau +\mathfrak{u}_1(\tau) -u)^2/2}du\right)dv'.
\end{eqnarray*}
By using the following obvious equality 
$$\|v'\|^2u + \big(\tau +\mathfrak{u}_1(\tau)-u\big)^2/2 = \frac{1}{2} \bigg[\big(u+\|v'\|^2-(\tau +\mathfrak{u}_1(\tau))\big)^2-\|v'\|^4 + 2\big(\tau +\mathfrak{u}_1(\tau)\big)\|v'\|^2\bigg],$$ 
and by setting $y'=\sqrt{2(\tau +\mathfrak{u}_1(\tau))}\,v'$ (for $\tau$ large enough), we obtain: 
\begin{eqnarray*}
J_{1}(\tau)/\sqrt{2\pi} &=&\int_{\R^{d-1}}e^{-i\langle \mathfrak{u}'(\tau),v'\rangle}e^{-(\tau +\mathfrak{u}_1(\tau))\|v'\|^2}\left(\int_0^{+\infty}e^{-\frac{1}{2}\big[u+\|v'\|^2-(\tau +\mathfrak{u}_1(\tau))\big]^2}du\right)dv'\\
&=& \int_{\R^{d-1}}e^{-i\langle \mathfrak{u}'(\tau),v'\rangle}e^{-(\tau +\mathfrak{u}_1(\tau))\|v'\|^2}\left(\int_{\|v'\|^2-(\tau +\mathfrak{u}_1(\tau))}^{+\infty}e^{-\frac{x^2}{2}}dx\right)dv'\\
&=& \big(2(\tau +\mathfrak{u}_1(\tau)\big)^{-\frac{d-1}{2}} \ \times \\
&\ & \quad \quad \quad 
\int_{\R^{d-1}} e^{-i\langle \frac{\mathfrak{u}'(\tau)}{\sqrt{2(\tau +\mathfrak{u}_1(\tau))}},y'\rangle \, - \, \frac{\|y'\|^2}{2}}\left(\int_{\frac{|y'|^2}{2(\tau +\mathfrak{u}_1(\tau))}-(\tau +\mathfrak{u}_1(\tau))}^{+\infty}e^{-\frac{x^2}{2}}dx\right)dy'.
\end{eqnarray*}
Finally, since $\tau +\mathfrak{u}_1(\tau)\sim_{\tau\to+\infty}\tau$ by (\ref{u-v-tau}), Lebesgue's theorem and $(\ref{Gauss})$ give 
\begin{eqnarray*}
\lim_{\tau\to +\infty}\tau^{(d-1)/2}J_{1}(\tau) &=& \sqrt{2\pi}\, 2^{-(d-1)/2}\int_{\R^{d-1}}e^{-i\langle 
\frac{\ell\,'}{\sqrt{2}},y'\rangle}\,e^{-\frac{\|y'\|^2}{2}}\left(\int_{-\infty}^{+\infty}e^{-\frac{x^2}{2}}dx\right)dy'\\
&=&\sqrt{2\pi}\, 2^{-(d-1)/2}\sqrt{2\pi}\,(\sqrt{2\pi})^{d-1}e^{-\|\ell\, '\|^2/4} \\
&=& 2\pi^{(d+1)/2}
e^{- \|\ell\, '\|^2/4}. 
\end{eqnarray*}
Hence the desired property for $J_{1}(\tau)$. 
\end{proof}
%
\section{Applications to additive functionals of Markov chains} \label{sect-markov}
In this section, $(X_n)_{n\geq 0}$ denotes a Markov chain with state space $(E,\cE)$, transition kernel $Q(x,dy)$ and initial distribution $\mu$. We assume that $Q$ admits an invariant probability measure, denoted by $\pi$. We denote by $\P_\mu$ the probability distribution of $(X_n)_{n\geq 0}$ with respect to the initial distribution $\mu$. The associated expectation is denoted by $\E_\mu$. Finally we consider a measurable function $\xi : E\r\R^d$ and we define the associated additive functional:  
\begin{equation} \label{Sn} 
\forall n\geq 1,\quad S_n = \xi(X_1)+\cdots+\xi(X_n).
\end{equation}
The next moment conditions on $\xi$ will ensure that $\xi$ is $\pi$-integrable. We can then define the first moment vector $\vec m$ of $\xi$ w.r.t.~to $\pi$. We assume that $\vec m$ is nonzero, that is  
\begin{equation} \label{moyenne}
\mbox{$\vec m=\E_\pi[\xi(X_0)] = \int_E\xi\, d\pi \neq 0.$}
\end{equation}
Set $\xi_c := \xi-\vec m$, and for each $n\geq1$ define the following random variable, taking values in the set of nonnegative symmetric $d\times d$-matrices:
$$\mbox{$S_{n,c}^{\otimes 2} = \sum_{k,\ell=1}^n \xi_c(X_k)\, \xi_c(X_\ell)^*$},$$ 
where $^*$ stands for the transposition. The next conditions on $\xi$ will also enable us to define the following nonnegative symmetric $d\times d$-matrix: 
\begin{equation} \label{sigma}
\Sigma = \vec m\cdot \vec m^* + \lim_n\frac{1}{n}\, \E_\mu[S_{n,c}^{\otimes 2}]. 
\end{equation}
Finally we use the following standard (Markov) nonlattice condition:     

\noindent {\bf Nonlattice Condition.}  \\ {\it 
We say that $\xi$ is nonlattice if there do not exist any $(b,H,A,\theta$) with $b\in\R^d$, $H\neq \R^d$ a closed subgroup in $\R^d$, $A\in\cE$ a $\pi$-full $Q$-absorbing \footnote{Recall that $A\in\cE$ is said to be $\pi$-full if $\pi(A)=1$, and $Q$-absorbing if $Q(a,A)=1$ for all $a\in A$.}set, and finally $\theta\, :\, E\r\R^d$ a bounded measurable function, such that we have }
$$\forall x\in A, \quad \xi(y)+\theta(y)-\theta(x)\in b+H\ \ Q(x,dy)-p.s.$$ 

Our first application concerns $\rho$-mixing Markov chains, namely: the strong ergodicity property (\ref{K1}) holds on the usual Lebesgue space $\L^2(\pi)$ (see \cite{rosen}). For instance, this condition is fulfilled if 
$(X_n)_{n\geq0}$ is ergodic, aperiodic and satisfies the so-called Doeblin condition. 

\noindent{\bf Hypothesis H$_\rho$.} {\it The Markov chain $(X_n)_{n\geq0}$ is $\rho$-mixing and  $\mu=\pi$.\footnote{The stationarity condition $\mu = \pi$ may be replaced with $d\mu = \phi d\pi$ provided that the density function $\phi$ is in $\L^{p}(\pi)$ for some $p>p_{\, \varepsilon_0}$, with  
$p_{\, \varepsilon_0}\in(1,+\infty)$ depending on $\varepsilon_0$. See \cite[cor.~4]{guiher} for details.} 
Moreover $\xi$ is nonlattice and satisfies the following moment condition  }
\begin{equation} \label{mom-rho} 
\exists \varepsilon_0>0,\quad \E_\pi[\|\xi(X_0)\|^{m_d+\varepsilon_0}] < +\infty.
\end{equation}

Our  second application concerns $V$-geometrically ergodic Markov chains, namely: given some unbounded function $V : E\r [1,+\infty[$, Property (\ref{K1}) holds on the  weighted supremum-normed space 
$(\cB_V,\|\cdot\|_V)$ composed of all the measurable functions $f : E\r\C$ satisfying : $\, \|f\|_V : = \sup_{x\in E} |f(x)|/V(x) <\infty$, see \cite{mey}. In particular we have $\pi(V)<\infty$. 

\noindent{\bf Hypothesis H$_V$.} {\it The Markov chain $(X_n)_{n\geq0}$ is $V$-geometrically ergodic and $\mu(V)<\infty$.\footnote{For instance this condition holds when $\mu=\pi$ or $\mu$ is the Dirac distribution $\delta_x$ at some $x\in E$.} Moreover $\xi$ is nonlattice and satisfies the following domination condition: }
\begin{equation} \label{mom-v-geo}  
\exists\, \varepsilon_0>0,\ \sup_{x\in\R^d}\frac{\|\xi(x)\|^{m_d+\varepsilon_0}}{V(x)} < \infty.
\end{equation}

Our last application concerns Lipschitz iterative models. Here we suppose that $(E,d)$ is a complete metric space in which every closed ball is compact. The space $E$ is equipped with its borel $\sigma$-algebra $\cE$. Consider a measurable space  $(G,\cG)$ and a sequence $\{\vartheta_n\}_{n\geq 1}$ of $G$-valued i.i.d.~random variables. Let $F : (E\times G,\cE\otimes \cG)\r (E,\cE)$ be jointly measurable and Lipschitz continuous in the first variable. Then, given $X_0$ an $E$-valued r.v.~independent of the sequence $\{\vartheta_n\}_{n\geq 1}$, the associated Lipschitz iterative model (LIM) is the sequence $(X_n)_{n\in\N}$ of r.v.~recursively defined as follows (starting from $X_0$):   
\begin{equation} \label{ifs-formule} 
\forall n\geq 1,\ \ \ X_n := F(X_{n-1},\vartheta_n).
\end{equation}
The LIM $(X_n)_{n\in\N}$ is said to be strictly contractive when 
\begin{equation} \label{ifs-cont}
\cC : = \sup_{(x,y)\in E^2,\, x\neq y} \frac{d\big(F(x,\vartheta_1),F(y,\vartheta_1)\big)}{d(x,y)} < 1\ \ \text{ almost surely.} 
\end{equation}
\noindent{\bf  Hypothesis H$_{Lim}$.} {\it The Markov chain $(X_n)_{n\in\N}$ is a strictly contractive LIM  satisfying 
\begin{equation} \label{moment-ite} 
\exists s\geq0,\ \exists \varepsilon_0>0,\quad \E\left[d\big(F(x_0,\vartheta_1),x_0\big)^{(s+1)m_d+\varepsilon_0}\, \right] < \infty 
\end{equation}
where $x_0$ is some point in $E$. Moreover we assume that $\E\big[(d(X_0,x_0)^{(s+1)m_d + \varepsilon_0}\big] < \infty$\footnote{This condition holds when $X_0\sim\delta_x$. Under Assumptions~(\ref{ifs-cont}) (\ref{moment-ite}), it holds for $X_0\sim\pi$, see \cite{duf,benda}.} and that $\xi$ is nonlattice and satisfies  }
\begin{equation} \label{xi-iterative} 
\exists S\geq 0,\ \forall (x,y)\in E\times E,\ \ \|\xi(x)-\xi(y)\| \leq S\, d(x,y)\, \big[1+d(x,x_0)+d(y,x_0)\big]^s.
\end{equation} 
\begin{cor} \label{cor-ex}
Assume that one of the set of Hypotheses {\bf H$_\rho$}, {\bf H$_V$} or {\bf H$_{Lim}$} holds. Then $\vec m$ in (\ref{moyenne}) is well-defined, the limit in (\ref{sigma}) exists, and $\Sigma$ is invertible. Moreover, for any set $A\in B(\R^d)$ whose boundary has zero Lebesgue-measure, we have: 
\begin{equation} \label{ren-f=1} 
\lim_{\tau \to +\infty}\,
(2\pi \tau)^{\frac{d-1}{2}}\sum_{n=1}^{+\infty}\P_\mu\big(S_n - \tau\, \vec m \in A\big) = 
\frac{L_d(A)}{(\det\Sigma)^{\frac{1}{2}}\|\Sigma^{-\frac{1}{2}}\,\, \vec m\|}.  
\end{equation}
\end{cor}
\begin{proof}{} We apply Theorem~\ref{ren-theo} with $f=1_E$. The fact that $\vec m$ in (\ref{moyenne}) is well-defined is obvious under Hypotheses {\bf H$_\rho$} or  {\bf H$_V$}. Under Hypothesis {\bf H$_{Lim}$}, we have $\|\xi(\cdot)\| \leq \|\xi(x_0)\| + S(1+d(\cdot,x_0))^{s+1}$, so that $\|\xi(\cdot)\|$ is $\pi$-integrable under Conditions~(\ref{ifs-cont}) (\ref{moment-ite}), see \cite{duf,benda}. The others conditions of Hypothesis~{\bf (H)} follow from the results proved in \cite{guiher} (centered case). Indeed, note that $\xi_c = \xi-\vec m$ is $\pi$-centered. Then the existence of $\Sigma_c : = \lim_n\frac{1}{n}\, \E_\mu[S_{n,c}^{\otimes 2}]$ is proved in \cite[Sect.~4]{guiher} under the three above  hypotheses. Moreover, since $\xi$ is nonlattice, so is $\xi_c$. The matrix $\Sigma_c$ is then definite from  \cite[p.~437]{fl}, thus so is $\Sigma$ in (\ref{sigma}). Now define $S_{n,c} = \xi_c(X_1)+\cdots+\xi_c(X_n)$. We have 
\begin{equation} \label{fc-cas-cent-dec}
\forall n\geq 1,\ \forall t\in\R^d,\quad \E_\mu[e^{i\langle t,S_n\rangle}] = e^{in\langle t,\, \vec m\rangle}\, \E_\mu[e^{i\langle t,S_{n,c}\rangle}].
\end{equation}
From \cite[Sect.~4]{guiher}, under anyone of Hypotheses {\bf H$_\rho$}, {\bf H$_V$} or {\bf H$_{Lim}$}, the sequence $(X_n,S_{n,c})_{n\geq0}$ satisfies Hypothesis~$\cR(m)$ for some $m>m_d$. Denote by $\lambda_c(\cdot)$, $L_c(\cdot)$, and $R_{n,c}(\cdot)$ the associated complex-valued functions\footnote{They are derived from the perturbation theorem due to Keller and Liverani \cite{keli}: $\lambda_c(t)$ is the perturbed eigenvalue of the Fourier operators associated with $(X_n,S_{n,c})_{n\geq0}$, while $L_c(\cdot)$ and $R_{n,c}(\cdot)$ are linked to some spectral projections.}. 
Then the sequence $(X_n,S_{n})_{n\geq0}$ satisfies Hypothesis~$\cR(m)$ with $\lambda(\cdot)$, $L(\cdot)$, and $R_{n}(\cdot)$ given by: 
\begin{equation} \label{spect-cas-cent-dec}
\lambda(t) := e^{i\langle t,\, \vec m\rangle}\lambda_c(t),\quad L(t) = L_c(t),\quad R_{n}(t) := e^{in\langle t,\, \vec m\rangle}\,R_{n,c}(t).
\end{equation}
The fact that $\sum_{n\geq1}R_{n}(\cdot)$ (resp.~$\sum_{n\geq1}\E[e^{i\langle t, S_{n}\rangle}]$) converges on the ball $B_R$ (resp.~on the annulus $K_{r,b}$) and defines a function in $\cC_b^m(B_R,\C)$ (resp.~in $\cC_b^m(K_{r,b},\C)$) can be easily derived from (\ref{spect-cas-cent-dec}) (resp.~(\ref{fc-cas-cent-dec})) and the spectral formulas given in \cite{guiher}. Finally we know that $\lambda_c(0)=1$, $\nabla\lambda_c(0) = 0$ (since $\xi_c$ is $\pi$-centered), and that $L_c(0)=1$ in case $f=1_E$. Thus $-i\nabla\lambda(0) = \vec m$, $Hess\, \lambda(0) = - (\Sigma_c + \vec m\cdot \vec m^*)$, and $L(0)=1$. Consequently Hypothesis~{\bf (H)} is fulfilled. 
\end{proof}

Actually, under each of hypotheses {\bf H$_\rho$}, {\bf H$_V$} or {\bf H$_{Lim}$},  Theorem~\ref{ren-theo} applies for a large class of nonnegative functions $f$ (use the refinements stated in \cite{guiher}). Moreover Corollary~\ref{cor-ex} can be extended to the lattice case, see \cite{denis-these}. 

The domination or moment condition on $\xi$ in Hypotheses {\bf H$_\rho$}, {\bf H$_V$} or {\bf H$_{Lim}$} involves the optimal order $m_d$ of the i.i.d.~case \cite{stam} (up to $\varepsilon_0>0$). Corollary~\ref{cor-ex} can be derived from the results of \cite{bab} but under much stronger moment conditions (the comparisons with \cite{bab} in terms of moment conditions are the same as in \cite{guiher}). 

For example, consider in $\R^3$ the autoregressive model $X_n = A X_{n-1} +\vartheta_n$, where $X_0,\vartheta_1,\vartheta_2,\ldots$ are $\R^3$-valued i.i.d.~random variables, and $A$ is a contractive matrix of order~3. Clearly, taking $d(x,y) = \|x-y\|$, the sequence $(X_n)_{n\geq0}$ is a strictly contractive LIM. Now let us consider $S_n = X_1+\ldots+X_n$ (i.~e.~$\, \xi(x)=x$). Condition~(\ref{xi-iterative}) is fulfilled with $S=1$ and $s=0$. Consequently, if we have $\E\big[\, \|X_0\|^{2 + \varepsilon_0} + \|\vartheta_1\|^{2+\varepsilon_0}\big] <+\infty$ for some $\varepsilon_0>0$,  
then (\ref{moment-ite}) holds and $X_0$ satisfies the moment condition stated in Hypothesis {\bf H$_{Lim}$}. 
Then, if $\xi(x)=x$ is non-lattice, we have (\ref{ren-f=1}) with $\vec m =  \E_\pi[X_0]$ and $\Sigma$  defined in (\ref{sigma}).

\noindent{\it Additional remarks in Markov setting.} \\
\indent Theorem~\ref{ren-theo} may be applied to general Markov random walks $(X_n,S_n)_{n\geq0}$, see (\ref{def-mrw}). Above we have only considered the special instance of additive functionals $S_n = \xi(X_1)+\cdots+\xi(X_n)$. For general MRWs, Hypothesis~$\cR(m)$ involves the increments $Y_n := S_n-S_{n-1}$ in place of the function $\xi$, e.g.~see \cite{DLJ} for MRWs associated with $\rho$-mixing driving Markov chains. In particular, for the three above Markov models, Hypothesis~$\cR(m)$ is investigated in \cite{DLJ, JLV} for bivariate additive functionals of type $S_n = \sum_{k=1}^n \psi(X_{k-1},X_k)$. \\
\indent Under some strong non-lattice conditions, asymptotic refinements of (\ref{ren-f=1}) have been obtained under the $V$-geometrical ergodicity assumption in \cite{fuhlai}, and under the uniform ergodicity assumption in  \cite{uchi}. By using the weak spectral method \cite{fl}, some results of \cite{fuhlai,uchi} could be improved in terms of moment conditions.  \\
\indent The notion of convergence cone (see \cite{bab}) is not investigated here. In fact, this study requires to define the Laplace kernels (in place of the Fourier kernels). The definition of such kernels needs some exponential (operator-type) moment conditions. Consequently the usual spectral method applies as stated in \cite{bab}.  

\appendix 

\section{Proof of Propositions~\ref{estimat-fourrier-q}-\ref{estimat-fourrier-q-bis}} \label{B}
This appendix completes and details some arguments summarized in \cite{bab}. Roughly speaking, the dyadic decomposition consists in writing an integral on $\R^d$ as the sum of integrals on the dyadic annuli $D_n:=\{2^n \leq \|x\| < 2^{n+1} \}$, $n\in\Z$. Here some asymmetric annuli (in place of $D_n$) must be considered to take  into account the mean direction $\vec{m}=\|\vec m\|\vec e_1$. Let us first specify some notations. 

For any real number $m>0$, we set $\tau := m - \lfloor m\rfloor$, and for any open subset ${\cal O}$ of $\R^d$, any $K\subset \cal O$ and $f\in \cC_b^m({\cal O},\C)$, we define 
$$\big\|f\big\|_{m,K}=\sum_{|\beta|\leq \lfloor m\rfloor}\big\|\partial^{\beta}f\big\|_{0,K}+\sum_{|\beta|=\lfloor m\rfloor}\big[\partial^{\beta}f\big]_{\tau,K}.$$
Let us recall that, for $x=(x_1,x_2,\ldots,x_d)\in\R^d$, we have set $x'=(x_2,\ldots,x_d)\in\R^{d-1}$ and  $w(x)=-ix_1+\|x'\|^2$. Next define the following: 
\begin{subequations}
\begin{eqnarray}
& & \text{for}\ 0<\omega_0<\omega'_0:\quad \Gamma_{0,\omega_0,\omega'_0}:=\big\{x\in\R^d\ : \ \omega_0\leq |w(x)|\leq \omega'_0\big\} \label{Gamma0}  \\
& & \forall k\in\Z,\ \forall x=(x_1,\ldots,x_d) \in\R^d,\quad 
D_k(x) := (\frac{x_1}{4^k},\frac{x_2}{2^k},\ldots,\frac{x_d}{2^k}) \label{D-k} \\
& & \forall k\in\Z,\ \Gamma_{k,\omega_0,\omega'_0}:=D_k(\Gamma_{0,\omega_0,\omega'_0})=\big\{x\in\R^d\ :\ \frac{\omega_0}{4^{k}}\leq |w(x)|\leq \frac{\omega'_0}{4^k}\big\}, \label{Gamma-k} \\
& & \forall k\in\Z,\ \Gamma_k:=\Gamma_{k,\frac{1}{4},1} = \big\{x\in\R^d\ :\ \frac{1}{4^{k+1}}\leq |w(x)|\leq \frac{1}{4^k}\big\} \label{defGammak} \\
& & \widetilde{\Gamma_0}:=\Gamma_{0,\frac{1}{8},2}\subset \Gamma_{-1}\cup\Gamma_0\cup\Gamma_1.
\end{eqnarray}
\end{subequations}
We will repeatedly use the following obvious inclusions: 
\begin{subequations}
\begin{eqnarray}
& & \forall k\in\N,\quad \Gamma_{k,\omega_0,\omega'_0}\subset B\big(0,(\omega'^2_0+\omega'_0)^{1/2}/2^k\big)
\label{inclusion Gammak} \\
& & \forall k\in \Z,\quad  D_k(\widetilde{\Gamma_0})\subset \Gamma_{k-1}\cup \Gamma_k\cup \Gamma_{k+1}. 
\label{inclusion DkGamma0tilde}
\end{eqnarray}
\end{subequations}

Now let us fix a real number $r>0$ and $k_0\in\N^*$ such that $\frac{\sqrt{2}}{2^{k_0-1}}< r$. For each $k\geq k_0-1$ we clearly have $\Gamma_{k}\subset B(0,\frac{\sqrt{2}}{2^{k}})\subset B(0,\frac{\sqrt{2}}{2^{k_0-1}})\subset B_{r}$. Moreover we have   
\begin{equation} \label{bis-Gamma-k-inclusion} 
\mbox{$\forall k\geq k_0,\quad D_k(\widetilde{\Gamma_0})\subset B\big(0,\sqrt{2}/2^{k-1}\big) \subset B_{r}$}. 
\end{equation}
For any function $u : B_{r} \r \C$, we define the following functions: 
$$\forall k\in \N,\ \forall x\in B_{r},\quad u_k(x)=u(D_kx).$$
Observe that, if $k\geq k_0$, then $u$ is defined on $D_k(\widetilde{\Gamma_0})$, so that $u_k$ is well-defined on $\widetilde{\Gamma_0}$. 
%
\subsection{Construction of a partition of the unity on  $\R^d\setminus\{0\}$} \label{b2}
Starting with $\widetilde {\Gamma_0}=\big\{x
\in\R^d\ :\ \frac{1}{8}\leq |w(x)|\leq 2\big\}$, let us define for all $k\in \Z$  
\begin{equation} \label{tile-gamma-k}
\mbox{$\widetilde{\Gamma_k}:=D_k(\widetilde{\Gamma_0})=\big\{x
\in\R^d\ :\ \frac{1}{8}\, \frac{1}{4^{k}}\leq |w(x)|\leq 2\, \frac{1}{4^k}\big\}$}.
\end{equation}
$\widetilde{\Gamma_k}$  contains $\Gamma_k$ (cf. (\ref{defGammak})). Now let $\gamma\in\cC_b^{\infty}(\R^d,\R^+)$ be  compactly supported in $\widetilde{\Gamma_0}$, such that 
$$\forall x\in \Gamma_0,\ \ \gamma(x)=1.$$
Note that $\gamma\circ D_{-k}\in\cC_b^{\infty}(\R^d,\R^+)$ is compactly supported in  $\widetilde{\Gamma_k}$. Next we set
$$\forall x\in\R^d,\quad \phi(x)=\sum_{k\in\Z}\gamma(D_{-k}x).$$ 
We have $\phi(0)=0$, and for $x\neq0$: 
$$\mbox{$x\in\widetilde{\Gamma_k}\ \ \Leftrightarrow\ \ k\in[\frac{-\ln8-\ln |w(x)|}{\ln4},\frac{\ln2-\ln |w(x)|}{\ln4}]$}.$$ 
Since the length of the previous interval is $2$, a point $x\neq0$ belongs at most to three sets among the $\widetilde{\Gamma_k}$'s, and since $\R^d\setminus\{0\}=\cup_{p\in\Z}\Gamma_p=\cup_{p\in\Z}\widetilde{\Gamma_p}$, we have: $1\leq\phi(x)<+\infty$. Notice also that we have by definition of $\phi$ 
\begin{equation}\label{propdephi}
\forall \ell\in\Z,\,\, \phi\circ D_\ell=\phi.
\end{equation}  
\begin{apro} \label{partition}
The positive function $\rho:= \gamma/\phi$ is infinitely differentiable on $\R^d\setminus\{0\}$. Moreover $\rho$ vanishes on $\R^d\setminus \big(\widetilde{\Gamma_0}\cup\{0\}\big)$, and we have: $\forall x\in\R^d\setminus\{0\},\  \sum_{k\in\Z}\rho(D_{-k}x) = 1$. 
\end{apro}
\begin{proof}{}
For $p\in \Z$, we denote by $Int\, (\Gamma_p\cup\Gamma_{p+1})$ the interior of $\Gamma_p\cup\Gamma_{p+1}$. Since $\displaystyle \R^d\setminus\{0\}=\cup_{p\in\Z}\, Int\, (\Gamma_p\cup\Gamma_{p+1})$, it suffices to check that $\phi$ is $\cC^{\infty}$ on the open subset $Int(\Gamma_p\cup\Gamma_{p+1})$. 
Let $x\in\, Int(\Gamma_p\cup\Gamma_{p+1})$, namely: $\frac{1}{4^{p+2}}<|w(x)|<\frac{1}{4^p}$, and let $k\leq p-2$. Since  $\frac{1}{4^p}\leq\frac{1}{16}\frac{1}{4^k}$, we have $|w(x)|<\frac{1}{8}\frac{1}{4^k}$. Thus $x\not\in \widetilde{\Gamma_k}$, and so $\gamma(D_{-k}(x))=0$. Similarly, if $k\geq p+3$, then $4\frac{1}{4^k}\leq\frac{1}{4^{p+2}}$, thus $|w(x)|>2\frac{1}{4^k}$, and so $x\not\in \widetilde{\Gamma_k}$ and  $\gamma(D_{-k}(x))=0$. It follows that the restriction of $\phi$ to $Int(\Gamma_p\cup\Gamma_{p+1})$ is the finite sum  $\sum_{k=p-1}^{p+2}\gamma\circ D_{-k}$. This proves the first point of Proposition~\ref{partition}. The two last assertions are obvious. 
\end{proof}
\subsection{Proof of Proposition~\ref{estimat-fourrier-q}} \label{b3}
Recall that $r>0$, that $k_0\in\N^*$ is such that $\frac{\sqrt{2}}{2^{k_0-1}}< r$, and that $\Gamma_k\subset B_{r}$ for all $k\geq k_0-1$. Define 
$${\cal C} := \cup_{j\geq k_0+1}\Gamma_j = \big\{x\in\R^d, 0<|w(x)|\leq 4^{-(k_0+1)}\big\}$$
and set $r'=2^{-1/2}\, 4^{-(k_0+1)}$. Then it can be easily seen that 
$B(0,r')\setminus\{0\}\subset {\cal C} \subset \overline{\cal C} \subset B_{r}$. Let us consider $\eta\in \cC_b^{\infty}(\R^d,\R)$, compactly supported in $\overline{\cal C}$, such that we have: 
$$\forall x\in B(0,r'),\quad \eta(x)=1.$$ 
Recall that $m_d:=\max(2,(d-1)/2)$. Let $m>m_d$ and let $\theta$ and $v$ be complex-valued functions on $B_{r}$ satisfying the hypotheses of Proposition~\ref{estimat-fourrier-q}. Recall that $q := 1_{B_{r}}\, \theta/v$. The function $q_1:=(1-\eta)q$ is in $\cC_b^{\lfloor m\rfloor}(\R^d,\C)$ and is compactly supported in $\overline{K}_{r',r}$. Moreover its restriction to $K_{r',r}$ is in $\cC_b^m(K_{r',r},\C)$. It follows from Proposition~\ref{Cm-elementaire-infini} applied with ${\cal O}=K_{r',r}$ and $u=q_1$ that $\lim_{\|a\|\to +\infty}\|a\|^{(d-1)/2}\hat{q}_1(a)=0$ (use $m>(d-1)/2$). 
 
{\it From the previous remark and the fact that $\eta\theta$ satisfies the same hypotheses as $\theta$, it suffices to prove Proposition~\ref{estimat-fourrier-q} in the case  when $q$ is compactly supported in $\overline {\cal C}$. From now on, in addition to the assumptions of Proposition~\ref{estimat-fourrier-q}, we assume that $q$ is compactly supported in $\overline{\cal C}$.} 

To prove that $q$ is integrable and to estimate $\hat{q}$, we use Proposition~\ref{partition}. Observe that $$D_{-k}({\cal C}) = \cup_{j\geq k_0+1}D_{-k}(\Gamma_j) = \cup_{\ell\geq k_0+1-k}\Gamma_\ell.$$ 
Hence, if $k\leq k_0-1$, then $D_{-k}({\cal C}) \subset \cup_{\ell\geq 2} \Gamma_\ell$. Since $\widetilde{\Gamma_0}$ and $\cup_{\ell\geq 2} \Gamma_\ell$ are disjoint, it follows from Proposition~\ref{partition} that, for each $k\leq k_0-1$, the function $\rho\circ D_{-k}$ vanishes on $\cal C$ . Moreover we have $q(x)=0$ if $x\not\in {\cal C}\cup\{0\}$. Thus  
$$\mbox{$\forall x\in\R^d\setminus\{0\},\ \ \ q(x) = \sum_{k= k_0}^{+\infty}\rho(D_{-k}x)\, q(x)$}.$$
For each $k\geq k_0$, we set: 
$$\forall x\in\R^d,\quad \psi_{k}(x) := \left \{
    \begin{array}{ll}
      \rho(x)\, q(D_{k}x) \quad \text{if } \ x\in\widetilde{\Gamma_0}\\[0.12cm]
       \ \quad 0 \quad \quad \quad \quad \  \text{if } \ x\notin\widetilde{\Gamma_0}.
    \end{array}
\right. 
= \ \left \{
    \begin{array}{ll}
      \rho(x)\, \frac{\theta_k(x)}{v_k(x)} \quad \text{if } x\in\widetilde{\Gamma_0}\\[0.12cm]
       \ \quad 0 \quad \quad \quad \ \text{if } x\notin\widetilde{\Gamma_0}.
    \end{array}
\right. 
$$
The following proposition is the key statement to prove Proposition~\ref{estimat-fourrier-q}. 
\begin{apro} \label{deri-theta-sur-v} 
For each $k\geq k_0$, we have $\psi_{k}\in\cC_b^m(\R^d,\C)$, and there exists $K>0$ such that: $\forall k\geq k_0,\ \|\psi_{k}\|_{m,\widetilde{\Gamma_0}}\leq K\, 2^k$.
\end{apro}
The proof of Proposition~\ref{deri-theta-sur-v} is postponed in Subsection~\ref{b5}. 

\noindent \begin{proof}{ of Proposition~\ref{estimat-fourrier-q}}
Let $k\geq k_0$. Setting  $t=D_{-k}x$ and using Proposition~\ref{deri-theta-sur-v}, we obtain 
\begin{equation} \label{q-int-estimation}
\int_{\R^d}\rho(D_{-k}x)\,|q(x)|dx = (\frac{1}{2^k})^{d+1}\int_{\R^d}\rho(t)\,|q(D_kt)|dt 
\leq \frac{KL_d(\widetilde{\Gamma_0})}{2^{kd}}. 
\end{equation}
Thus $\sum_{k\geq k_0}\int_{\R^d}\rho(D_{-k}x)\,|q(x)|dx < \infty$. So $q$ is integrable, and we have for all $a\in\R^d$: 
\begin{eqnarray} \label{hat-q}
\hat{q}(a) := \int_{\R^d} q(x)\, e^{-i\langle x,a \rangle}\, dx &=&\sum_{k=k_0}^{+\infty}\int_{\R^d}\rho(D_{-k}x)\,q(x)e^{-i\langle x,\,a\rangle}dx \nonumber \\
&=&\sum_{k=k_0}^{+\infty}(\frac{1}{2^k})^{d+1}\int_{\R^d}\rho(t)\,q(D_{k}t)e^{-i\langle t,\,D_ka\rangle}dt \nonumber\\
&=&\sum_{k=k_0}^{+\infty}(\frac{1}{2^k})^{d+1}\widehat{\psi_{k}}(D_ka). 
\end{eqnarray}
From Proposition~\ref{deri-theta-sur-v} and Proposition~\ref{Cm-elementaire-infini} applied with $u:=\psi_{k}$ and ${\cal O} := Int(\widetilde{\Gamma_0})$, 
one can deduce the following property. 
\begin{equation} \label{psi-k-rho}
\forall k \geq k_0,\ \forall b\in\R^d,\quad \|b\|^{m}|\widehat{\psi_{k}}(b)|\leq CK\, 2^k.
\end{equation}
Moreover Proposition~\ref{deri-theta-sur-v} gives with $K' := K L_d(\widetilde{\Gamma_0})$: $\forall k \geq k_0,\ \forall b\in\R^d,\ |\widehat{\psi_{k}}(b)|\leq K'\, 2^k$. Then, from (\ref{psi-k-rho}) and $m> \frac{d-1}{2}$, we obtain  
$$\mbox{$\forall k \geq k_0,\ \forall b\in\R^d,\quad \|b\|^{\frac{d-1}{2}}|\widehat{\psi_{k}}(b)|\leq \big(\|b\|^m +1\big)|\widehat{\psi_{k}}(b)|\leq (CK+K')\, 2^k$}.$$
Let $a\in\R^d$. By using the fact that $\|a\| \leq 4^k\|D_ka\|$, we obtain for all $k\geq k_0$  
\begin{equation} \label{hat-psi}
\|a\|^{\frac{d-1}{2}}(\frac{1}{2^k})^{d+1}|\widehat{\psi_{k}}(D_ka)| \leq  
(4^k)^{\frac{d-1}{2}}\, (\frac{1}{2^k})^{d+1}\, (CK+K')\, 2^k = \frac{CK+K'}{2^k}, 
\end{equation}
from which we deduce that 
$$\lim_{k_1\r+\infty} 
\mbox{$ \|a\|^{\frac{d-1}{2}}\,\big|\sum_{k=k_1}^{+\infty}(\frac{1}{2^k})^{d+1}\widehat{\psi_{k}}(D_ka)\, \big| = 0\ $ uniformly in $a\in\R^d$}.$$
By (\ref{psi-k-rho}), we have $\forall a\in\R^d\setminus\{0\},\ |\widehat{\psi_{k}}(D_ka)|\leq \frac{4^{km}C2^k}{\|a\|^{m}}$, and since $m>\frac{d-1}{2}$, we obtain 
$$\forall k_1 > k_0,\quad \lim_{\|a\|\to+\infty}
\mbox{$\|a\|^{\frac{d-1}{2}}\sum_{k=k_0}^{k_1-1}(\frac{1}{2^k})^{d+1}\widehat{\psi_{k}}(D_ka)=0$}.$$
From (\ref{hat-q}) and from the two previous properties, it follows that $\lim_{\|a\|\to +\infty}\|a\|^{\frac{d-1}{2}}\hat{q}(a)=0$, as claimed in Proposition~\ref{estimat-fourrier-q}. 
\end{proof}

\subsection{Proof of Proposition~\ref{estimat-fourrier-q-bis}} \label{b4}
Here the function $q$ is defined by $q := 1_{B_{r}}\, \theta/(v\,\tilde{v})$, with $\theta$, $v$ and $\tilde{v}$ satisfying the hypotheses of Proposition~\ref{estimat-fourrier-q-bis}. As above one may suppose that $q$ is compactly supported in $\overline {\cal C}$. The proof of Proposition~\ref{estimat-fourrier-q-bis} is then similar to the previous one, up to the following changes. First the function $\psi_{k}(\cdot)$ is replaced with:  
	$$\forall x\in\R^d,\quad \widetilde\psi_{k}(x):= \left \{
    \begin{array}{ll}
      \rho(x)\, \theta_k(x)/(v_k(x)\, \tilde{v}_k(x)) \quad \text{if } x\in\widetilde{\Gamma_0}\\[0.2cm]
       \ \quad 0 \quad \quad \text{if } x\notin\widetilde{\Gamma_0}. 
    \end{array}
\right. 
$$
Second, Proposition~\ref{deri-theta-sur-v} is replaced with the following one (proved in Subsection~\ref{b5}): 
\begin{apro} \label{deri-theta-sur-v-bis}
For each $k\geq k_0$, we have $\widetilde\psi_{k}\in\cC_b^m(\R^d,\C)$, and  there exists $L>0$ such that, for all $k\geq k_0$, we have $\|\widetilde\psi_{k}\|_{m,\widetilde{\Gamma_0}}\leq L\, (2^k)^{2-\nu}$ with $\nu := \min(m-2,1)$. 
\end{apro}
Note that $\nu>0$ since $m>m_d\geq2$ by hypothesis. Next the term $O(2^{-kd})$ in (\ref{q-int-estimation}) is replaced with $O(2^{-k(d-1+\nu)})$: this yields the integrability of $q$ and a formula analogous to (\ref{hat-q}). The term $O(2^k)$ of (\ref{psi-k-rho}) is replaced with $O(2^{k(2-\nu)})$: this gives the property analogous to (\ref{hat-psi}) with $O(2^{-k\nu})$ (in place of $O(2^{-k})$). We can then conclude as above. 
\fdem

\subsection{Proof of Propositions~\ref{deri-theta-sur-v} and \ref{deri-theta-sur-v-bis}} \label{b5}
For any $\alpha=(\alpha_1,\ldots,\alpha_d)\in\N^d$, we set $\alpha! := \alpha_1!\ldots\alpha_d!$. If $\gamma=(\gamma_1,\ldots,\gamma_d)\in\N^d$, the notation $\alpha\leq \gamma$ means: $\forall i=1,\ldots,d,\ \alpha_i\leq\gamma_i$.
We shall use Leibniz's formula, namely:  if $\Omega$ is an open subset of $\R^d$ and if  $\gamma\in\N^d$ is such that $\vert\gamma\vert := \gamma_1+\ldots+\gamma_d \leq k\, $ ($k\in\N^*$), then we have for all $f,g\in  \cC_b^k(\Omega,\,\C)$ 
\begin{equation} \label{def-leibniz}
\mbox{$\partial^{\gamma}(f\cdot g)=\sum_{\beta\leq\gamma} \binom{\gamma}{\beta}\ \partial^{\beta}f \ \partial^{\gamma-\beta}g$},
\end{equation}
where $\binom{\gamma}{\beta} =\frac{\gamma!}{\beta!(\gamma-\beta)!}$  We shall also use repeatedly the following lemma. 
\begin{alem} \label{preli} 
Let $\sigma\in]0,1]$ and let ${\cal O}$ be an open subset of $\R^d$. For any $f,g\in \cC_b^{\sigma}(\cal O,\C)$, we have  $fg\in \cC_b^{\sigma}(\cal O,\C)$ and $\big[fg\big]_{\sigma,\cal O}\leq \big\|f\big\|_{0,\cal O}\, \big[g\big]_{\sigma,\cal O} + \big\|g\big\|_{0,\cal O}\, \big[f\big]_{\sigma,\cal O}$. 
\end{alem}
Recall that $\widetilde {\Gamma_0}=\big\{x\in\R^d:1/8\leq |w(x)|\leq 2\big\}$ and that, for any function $u : B_{r} \r \C$, we have set: $\forall k\in \N,\, \forall x\in B_{r},\ u_k(x)=u(D_kx)$. Observe that all the partial derivatives of the function $\rho(\cdot)$ are bounded on $\widetilde{\Gamma_0}$. From this fact and from (\ref{def-leibniz}), 
Propositions~\ref{deri-theta-sur-v} and \ref{deri-theta-sur-v-bis} easily follow from the next Lemmas~\ref{deri-theta}, \ref{deri-sec-theta} and \ref{deri-1/v}. 

Below we denote by $\theta, v, \tilde{v}$ three functions from $B_{r}$ into $\C$, we consider a real number $m>2$, and we set $\tau := m - \lfloor m\rfloor$. 
\begin{alem} \label{deri-theta} 
If $\theta\in \cC_b^m(B_{r},\C)$ satisfies $\theta(0)=0$, then there exist a constant $C>0$ such that, for all $k\geq k_0$, we have
\begin{subequations}
\begin{eqnarray}
|\beta|\leq \lfloor m\rfloor\quad &\Rightarrow& \quad  \big\|\partial^{\beta}\theta_k\big\|_{0,\widetilde{\Gamma_0}}\leq C\, 2^{-k}  \label{deri-theta1} \\
|\beta|=\lfloor m\rfloor\quad &\Rightarrow& \quad \big[\partial^{\beta}\theta_k\big]_{\tau,\widetilde{\Gamma_0}}\leq C\, 2^{-k} \label{deri-theta3}
\end{eqnarray}
\end{subequations}
\end{alem}
\begin{proof}{} 
Since the first partial derivatives of $\theta$ are bounded on $B_{r}$, there exists $M>0$ such that: $\forall x\in B_{r}$, $|\theta(x)|\leq M\|x\|$. From (\ref{bis-Gamma-k-inclusion}), it follows that $\|\theta_k\|_{0,\widetilde{\Gamma_0}}\leq M\sqrt{2}/2^{k-1}$. 
Let  $\beta\in\N^d$ such that $1\leq|\beta|\leq \lfloor m\rfloor$, and let $x\in B_{r}$. We have $|\partial^{\beta}\theta_k(x)|\leq \frac{1}{2^{k|\beta|}}|\partial^{\beta}\theta(D_kx)|$. So (\ref{deri-theta1}) follows from the fact that the partial derivatives of $\theta$ of order $j=1,\ldots,\lfloor m\rfloor$ are bounded on $B_{r}$. Now assume that $|\beta|=\lfloor m\rfloor\geq2$. Then we have for all $(x,y)\in \widetilde{\Gamma_0}\times\widetilde{\Gamma_0}$:  
$$\mbox{$\big|(\partial^{\beta}\theta_k)(x)-(\partial^{\beta}\theta_k)(y)\big|\leq 4^{-k}\, \big|(\partial^{\beta}\theta)(D_kx)-(\partial^{\beta}\theta)(D_ky)\big| 
\leq 4^{-k}\, \big[\partial^{\beta}\theta\big]_{\tau,B_{r}}\, \|x-y\|^{\tau}$},$$
hence we have (\ref{deri-theta3}). 
\end{proof}
%
\begin{alem} \label{deri-sec-theta} 
If $\theta\in \cC_b^m(B_{r},\C)$ satisfies $\theta(0) = 0$ and if all the first and second partial derivatives of $\theta$ vanish at 0, then there exist a constant $D>0$ such that, for all $k\geq k_0$, we have with $\nu := \min(m-2,1)$: 
\begin{subequations}
\begin{eqnarray}
|\beta|\leq \lfloor m\rfloor\quad &\Rightarrow& \quad \big\|\partial^{\beta}\theta_k\big\|_{0,\widetilde{\Gamma_0}} 
\leq D\, 2^{-k(2+\nu)} \label{deri-sec-theta1} \\
|\beta|=\lfloor m\rfloor\quad &\Rightarrow& \quad \big[\partial^{\beta}\theta_k\big]_{\tau,\widetilde{\Gamma_0}} \leq D\, 2^{-k(2+\nu)}. \label{deri-sec-theta3}
\end{eqnarray}
\end{subequations}
\end{alem}
\begin{proof}{}
By hypothesis we have  $\theta(0)=0$, $(\partial_j\theta)(0)=0$, and $(\partial^2_{j\ell}\theta)(0)=0$ for every  $(j,\ell)\in\{1,\ldots,d\}^2$. 
If $|\beta|\leq2$, then (\ref{deri-sec-theta1}) holds since there exists a constant $K_{\beta}>0$ such that~: $\forall x\in B_{r},\ |\partial^{\beta}\theta(x)| \leq K_{\beta}\, \|x\|^{2+\nu-|\beta|}$. For  $2<|\beta|\leq \lfloor m\rfloor$, (\ref{deri-sec-theta1}) is obvious. 
Now assume that $|\beta|=\lfloor m\rfloor$. Let $(x,y)\in \widetilde{\Gamma_0}^2$. We have 
\begin{eqnarray*}
|(\partial^{\beta}\theta_k)(x)-(\partial^{\beta}\theta_k)(y)|&\leq& 2^{-k\lfloor m\rfloor}\, |(\partial^{\beta}\theta)(D_kx)-(\partial^{\beta}\theta)(D_ky)|\\
&\leq& 2^{-k\lfloor m\rfloor}\,\big[\partial^{\beta}\theta\big]_{\tau,B_{r}}(\frac{1}{2^k})^{\tau}\|x-y\|^{\tau}. 
\end{eqnarray*}
Since $\lfloor m\rfloor + \tau = m$, this gives (\ref{deri-sec-theta3}). 
\end{proof}
%
\begin{asublem} \label{deri-v} 
If $v\in \cC^m(B_{r},\C)$ satisfies (\ref{(i)}) and (\ref{(ii)}), then there exist $c,c',c''>0$ such that, for all $k\geq k_0$, we have: 
\begin{subequations}
\begin{eqnarray}  
|\beta|\leq \lfloor m\rfloor \quad &\Rightarrow& \quad \big\|\partial^{\beta}v_k\big\|_{0,\widetilde{\Gamma_0}}\leq 
c\, 4^{-k} \label{deri-v1} \\
|\beta|\leq \lfloor m\rfloor-1\quad &\Rightarrow& \quad\big[\partial^{\beta}v_k\big]_{1,\widetilde{\Gamma_0}}\leq c'\, 4^{-k}   \label{deri-v2}  \\
|\beta|=\lfloor m\rfloor\quad &\Rightarrow& \quad \big[\partial^{\beta}v_k\big]_{\tau,\widetilde{\Gamma_0}}\leq 
c''\, 4^{-k} \label{deri-v3} 
\end{eqnarray}
\end{subequations}
\end{asublem}
\begin{proof}{}
From (\ref{(ii)}), we obtain $|v_k(x)| \leq b\, |w(D_kx)| \leq b\, 4^{-k} |w(x)|$, hence  $\big\|v_k\big\|_{0,\widetilde{\Gamma_0}}\leq 2b\, 4^{-k}$. Moreover we have $\forall x\in B_{r}$, $(\partial_1 v_k)(x)= 4^{-k}(\partial_1v)(D_kx)$, thus $\|\partial_1v_k\|_{0,\widetilde{\Gamma_0}}\leq 4^{-k}\big\|\partial_1v\big\|_{0,B_{r}}$. Next we have: $\forall j\in\{2,\ldots,d\},\ \forall x\in B_{r},\ (\partial_j v_k)(x)=2^{-k}(\partial_jv)(D_kx)$. From (\ref{(i)}) and the mean value inequality (notice that the second order partial derivatives of $v$ are bounded on  $B_{r}$), there exists $M>0$ 
such that: $\forall x\in B_{r},\, | (\partial_j v)(x)|\leq M \|x\|$.  
From (\ref{bis-Gamma-k-inclusion}), it follows that $\|\partial_jv_k\|_{0,\widetilde{\Gamma_0}}\leq M\, 2\sqrt{2}\, 4^{-k}$. This proves (\ref{deri-v1}) for $|\beta|=1$. Now, if $2\leq|\beta|\leq \lfloor m\rfloor$, then  $$\mbox{$\forall x\in B_{r},\quad (\partial^{\beta}v_k)(x)=(\frac{1}{2^k})^{\beta_1}(\frac{1}{4^k})^{\beta_2+\ldots+\beta_d}(\partial^{\beta}v)(D_kx)$},$$  
and therefore we have $\big\|\partial^{\beta}v_k\big\|_{0,\widetilde{\Gamma_0}}\leq \frac{1}{4^k}\big\|\partial^{\beta}v\big\|_{0,B_{r}}$. The proof of (\ref{deri-v1}) is then complete. 

\noindent Let us first prove (\ref{deri-v2}) in case $\beta=0$. Set $V(x)=v(x)-(\partial_1v)(0)\, x_1$ for $x\in B_{r}$. We have $(\partial_1V)(0)=0$, and $(\partial_2V)(0)= \cdots=(\partial_dV)(0)=0$ thanks to (\ref{(i)}). So there exists $M>0$ such that: $\forall j\in\{1,\ldots,d\},\ \forall x\in B_{r},\ |(\partial_j V)(x)|\leq M\, \|x\|$. This fact and the mean value inequality applied on $B(0,\frac{\sqrt{2}}{2^{k-1}})$ imply that there exists $C'>0$ such that we have: $\forall (x,y)\in \widetilde{\Gamma_0}^2,\ |V(D_kx)-V(D_ky)|\leq \frac{C'}{2^k}\|D_kx-D_ky\|$. Since $v_k(x) = V(D_kx) + (\partial_1v)(0)\frac{x_1}{4^k}$, we obtain $|v_k(x)-v_k(y)|\leq \frac{C''}{4^k}\|x-y\|$ for some $C''>0$. Thus  $\big[v_k\big]_{1,\widetilde{\Gamma_0}}\leq \frac{C''}{4^k}$. 

\noindent Now we establish (\ref{deri-v2}) in case $|\beta|\in\{1,\ldots,\lfloor m\rfloor -1\}$. Since all the partial derivatives of order $|\beta| +1$ of $v$ are bounded on $B_{r}$, there exists $M>0$ such that we have for all $(x,y)\in \widetilde{\Gamma_0}^2$: 
\begin{eqnarray*}
|(\partial^{\beta}v_k)(x)-(\partial^{\beta}v_k)(y)|\leq\frac{1}{2^{|\beta|k}}|(\partial^{\beta}v)(D_kx)-(\partial^{\beta}v)(D_ky)| &\leq& M 2^{-|\beta|k} \|D_kx-D_ky\|\\
&\leq& M 4^{-k} \|x-y\|. 
\end{eqnarray*}
This yields (\ref{deri-v2}). The proof of (\ref{deri-v3}) is similar to that of (\ref{deri-theta3}). 
\end{proof}
\begin{alem} \label{deri-1/v}
If $v\in \cC_b^m(B_{r},\C)$ satisfies (\ref{(i)}) and (\ref{(ii)}), then there exist a constant $E>0$ such that, for all $k\geq k_0$, we have  
\begin{subequations}
\begin{eqnarray}
|\beta|\leq \lfloor m\rfloor \quad &\Rightarrow& \quad \big\|\partial^{\beta}(1/v_k)\big\|_{0,\widetilde{\Gamma_0}}\leq E4^k \label{deri-1/v1} \\
|\beta|=\lfloor m\rfloor \quad &\Rightarrow& \quad \big[\partial^{\beta}(1/v_k)\big]_{\tau,\widetilde{\Gamma_0}}\leq E4^k. \label{deri-1/v3}
\end{eqnarray}
\end{subequations}
\end{alem}
\begin{proof}{}
From (\ref{(ii)}) and (\ref{tile-gamma-k}), we have  $\|1/v_k\|_{0,\widetilde{\Gamma_0}} \leq \frac{2.4^{k+1}}{a}$. Let $j\in\{1,\ldots,d\}$. From (\ref{deri-v1}) and the previous inequality, we obtain  
$$\mbox{$\big\|\partial_j(\frac{1}{v_k})\big\|_{0,\widetilde{\Gamma_0}} \leq \big(\big\|\partial_jv_k\big\|_{0,\widetilde{\Gamma_0}}\big) \big(\big\|\frac{1}{v_k}\big\|_{0,\widetilde{\Gamma_0}}^2\big) \leq \frac{c}{4^k}\, \frac{4^{2k+3}}{a^2}=\frac{64c}{a^2}4^k$}.$$
Now, let us proceed by induction. Let $\ell\in\{1,\ldots,\lfloor m\rfloor-1\}$, and assume that, for all $\beta\in\N^d$ such that $|\beta|\leq \ell$, there exists $C_{\beta}>0$ such that 
$\big\|\partial^{\beta}(\frac{1}{v_k})\big\|_{0,\widetilde{\Gamma_0}}\leq C_{\beta}\, 4^k$. 
Let $\gamma\in\N^d$ such that $|\gamma|=\ell+1$. Since $\partial^{\gamma}(v_k^{-1}\,.\,v_k)=0$, Leibniz's formula gives:
\begin{equation} \label{Leibniz} 
\mbox{$\partial^{\gamma}(\frac{1}{v_k}) = -\frac{1}{v_k} \sum_{\beta\leq\gamma,\, \beta\neq\gamma} 
\binom{\gamma}{\beta}\partial^{\beta}(\frac{1}{v_k}).\,\partial^{\gamma-\beta}(v_k)$}.
\end{equation}
Thus we have (use the induction hypothesis and (\ref{deri-v1})): 
\begin{eqnarray*}
\mbox{$
\big\|\partial^{\gamma}(\frac{1}{v_k})\big\|_{0,\widetilde{\Gamma_0}}$} &\leq & \mbox{$\big\|\frac{1}{v_k}\big\|_{0,\widetilde{\Gamma_0}} \sum_{\beta\leq\gamma,\, \beta\neq\gamma} \binom{\gamma}{\beta} 
\big(\big\|\partial^{\beta}(\frac{1}{v_k})\big\|_{0,\widetilde{\Gamma_0}}\big) 
\big(\big\|\partial^{\gamma-\beta}(v_k)\big\|_{0,\widetilde{\Gamma_0}}\big)$} \\
&\leq& \mbox{$\frac{2.4^{k+1}}{a}\sum_{\beta\leq\gamma,\, \beta\neq\gamma} \binom{\gamma}{\beta}C_{\beta}4^k\frac{c}{4^k}$} \\
&\leq& \mbox{$\big(\frac{8c}{a}\sum_{\beta\leq\gamma,\, \beta\neq\gamma} \binom{\gamma}{\beta}C_{\beta}\big)\, 4^k.$}
\end{eqnarray*}
The proof of (\ref{deri-1/v1}) is complete. To prove (\ref{deri-1/v3}), we need the following.  
\begin{asublem} \label{sub-lem-1/v}
Under the hypotheses of Lemma~\ref{deri-1/v}, there exist a constant $c>0$ such that, for all $k\geq k_0$, we have
\begin{equation} \label{deri-1/v2}
|\beta|\leq \lfloor m\rfloor-1 \quad \Rightarrow \quad  \big[\partial^{\beta}(1/v_k)\big]_{1,\widetilde{\Gamma_0}}\leq c4^k. 
\end{equation}  
\end{asublem}

\noindent\begin{proof}{} 
If $\beta=0$, then we have for all $(x,y)\in \widetilde{\Gamma_0}^2$ (use (\ref{deri-v2}) and proceed as for (\ref{deri-1/v1})) $$|\frac{1}{v_k(x)}-\frac{1}{v_k(y)}| = \frac{|v_k(x)-v_k(y)|}{|v_k(x)||v_k(y)|} \leq\frac{1}{a^2}4^{2k+3} \big[v_k\big]_{1,\widetilde{\Gamma_0}}\|x-y\| \leq \frac{64c'}{a^2}4^{k} \|x-y\|.$$ Hence $[\frac{1}{v_k}]_{1,\widetilde{\Gamma_0}} \leq 64c'4^{k}/a^2$. Now, let us consider the case $|\beta|=1$. Let $j\in\{1,\ldots,d\}$. By using Lemma~\ref{preli}, (\ref{deri-v1}) (\ref{deri-1/v1}) and the previous inequality, we obtain  
\begin{eqnarray*}
\big[\partial_j\frac{1}{v_k}\big]_{1,\widetilde{\Gamma_0}} = \big[(\partial_jv_k)\cdot\frac{1}{v_k^2}\big]_{1,\widetilde{\Gamma_0}} 
&\leq& \big\|\partial_jv_k\big\|_{0,\widetilde{\Gamma_0}}\, \big[\frac{1}{v_k^2}\big]_{1,\widetilde{\Gamma_0}} + \big[\partial_jv_k\big]_{1,\widetilde{\Gamma_0}}\, \big\|\frac{1}{v_k^2}\big\|_{0,\widetilde{\Gamma_0}}\\
&\leq& \big\|\partial_jv_k\big\|_{0,\widetilde{\Gamma_0}}\, \big(2\, \big\|\frac{1}{v_k}\big\|_{0,\widetilde{\Gamma_0}}\, \big[\frac{1}{v_k}\big]_{1,\widetilde{\Gamma_0}}\big) + \big[\partial_jv_k\big]_{1,\widetilde{\Gamma_0}}\, \big(\big\|\frac{1}{v_k}\big\|_{0,\widetilde{\Gamma_0}}\big)^2 \\
&\leq&2\frac{c}{4^k}d4^k\frac{64c'}{a^2}4^{k}+\frac{c'}{4^k}(d4^k)^2:=c'''4^k. 
\end{eqnarray*}
This gives (\ref{deri-1/v2}) for $|\beta|=1$. To complete the proof of (\ref{deri-1/v2}), let us again proceed by induction. Assume that, for some $\ell\in\{1,\ldots,\lfloor m\rfloor-2\}$ and for all $\beta\in\N^d$ such that  $|\beta|\leq \ell$, there exists $D_{\beta}>0$ such that 
\begin{equation} \label{rec} 
\mbox{$\big[\partial^{\beta}(\frac{1}{v_k})\big]_{_{1,\widetilde{\Gamma_0}}}\leq D_{\beta}4^k$}.
\end{equation}
Let $\gamma\in\N^d$ such that $|\gamma|=\ell+1$, and let $\beta\leq\gamma$ such that $\beta\neq\gamma$. By applying Lemma~\ref{preli}, we get (below $\|\cdot\|_{0}$ and $[\cdot\,]_{1}$ stand for $\|\cdot\|_{0,\widetilde{\Gamma_0}}$ and  $[\cdot\,]_{1,\widetilde{\Gamma_0}}$ respectively) : 
\begin{eqnarray*} 
\big[\frac{1}{v_k}\cdot\,\partial^{\beta}(\frac{1}{v_k})\cdot\partial^{\gamma-\beta}(v_k)\big]_1 &\leq& \big[\frac{1}{v_k}\big]_1\,\big\|\partial^{\beta}(\frac{1}{v_k})\cdot\partial^{\gamma-\beta}(v_k)\big\|_0 + \big\|\frac{1}{v_k}\big\|_0\,\big[\partial^{\beta}(\frac{1}{v_k})\cdot\partial^{\gamma-\beta}(v_k)\big]_1\\
&\leq&\big[\frac{1}{v_k}\big]_1\,\big\|\partial^{\beta}(\frac{1}{v_k})\big\|_0\, \big\|\partial^{\gamma-\beta}(v_k)\big\|_0 \\ 
&\ & + \ \  \big\|\frac{1}{v_k}\big\|_0\bigg(\big[\partial^{\beta}(\frac{1}{v_k})\big]_1\big\|\partial^{\gamma-\beta}(v_k)\big\|_0 + 
\big\|\partial^{\beta}(\frac{1}{v_k})\big\|_0\,\big[\partial^{\gamma-\beta}(v_k)\big]_1\bigg). 
\end{eqnarray*}
From (\ref{deri-v1}) (\ref{deri-v2}) (\ref{deri-1/v1}) and (\ref{rec}) (observe that $|\beta|\leq \ell$), we get    
$\big[\frac{1}{v_k}\, \partial^{\beta}(\frac{1}{v_k})\, \partial^{\gamma-\beta}(v_k)\big]_1\leq L_\beta\, 4^k$ 
for some $L_\beta>0$. From (\ref{Leibniz}), it follows that $[\partial^{\gamma}(\frac{1}{v_k})]_{_{1,\widetilde{\Gamma_0}}}\leq D_{\gamma}4^k$ for some $D_{\gamma}>0$. 
\end{proof}

Finally we prove (\ref{deri-1/v3}). Let $\gamma\in\N^d$ be such that $|\gamma|=\lfloor m\rfloor$. If  $\beta\leq\gamma,\,\beta\neq 0, \beta\neq\gamma$ (thus $|\beta|\leq \lfloor m\rfloor-1$ and $|\gamma-\beta|\leq \lfloor m\rfloor-1$), we have $\big[\frac{1}{v_k}\, \partial^{\beta}(\frac{1}{v_k})\, \partial^{\gamma-\beta}(v_k)\big]_{1,\widetilde{\Gamma_0}} \leq E\, 4^k$ for some $E>0$ (use Lemma~\ref{preli} as above and (\ref{deri-v1}) (\ref{deri-v2}) (\ref{deri-1/v1}) (\ref{deri-1/v2})). Set 
$$s_k= -\frac{1}{v_k} \sum_{\beta\leq\gamma,\,\beta\neq 0, \beta\neq\gamma} 
\binom{\gamma}{\beta}\partial^{\beta}(\frac{1}{v_k})\,\partial^{\gamma-\beta}(v_k).$$
The previous remark shows that $[s_k]_{1,\widetilde{\Gamma_0}} \leq E'\, 4^k$ for some $E'>0$. The same inequality holds for $[s_k]_{\tau,\widetilde{\Gamma_0}}$ because $\widetilde{\Gamma_0}$ is bounded. Finally, from (\ref{Leibniz}), we get $\partial^{\gamma}(\frac{1}{v_k}) = -\frac{1}{v_k^2}\,\partial^{\gamma}v_k+s_k$, and Lemma~\ref{preli} and  (\ref{deri-1/v1}) (\ref{deri-v3}) give $[\partial^{\gamma}v_k/v_k^2]_{_{\tau,\widetilde{\Gamma_0}}} \leq E''\, 4^k$ for some $E''>0$. 
\end{proof}
\begin{arem} \label{rk-m}
Lemmas \ref{deri-theta} and \ref{deri-1/v}, and so Proposition~\ref{deri-theta-sur-v}, hold when $m=1$. Consequently the conclusion of Proposition~\ref{estimat-fourrier-q} is fulfilled as soon as $m > \max(1,(d-1)/2)$. However the condition $m > \max(2,(d-1)/2)$ (i.e.~$m > m_d$) seems to be necessary to prove Proposition~\ref{estimat-fourrier-q-bis}. 
\end{arem} 


\begin{thebibliography}{99}

\bibitem{bab2}
{\sc Babillot M.}
{\em Le noyau potentiel des  cha\^{\i}nes  semi-markoviennes.}  
Thèse de troisième cycle, Université de Paris-VII, 1985.

\bibitem{bab}
{\sc Babillot M.}
{\em Th\'eorie du renouvellement pour des  cha\^{\i}nes  semi-markoviennes transientes.}  
Ann. I. H. Poincar\'e, sect. B, Tome 24, No 4, 507-569 (1988). 

\bibitem{benda}
{\sc Benda M.}
{\em A central limit theorem for contractive stochastic dynamical systems.}  
J.~App.~Prob. {\bf 35} (1998) 200-205.

\bibitem{bre}
{\sc Breiman L.}
{\em Probability} Classic in Applied Mathematics, SIAM, 1993. 

\bibitem{broi}
{\sc Dal'bo F., Peign\'e M.}
{\em Comportement asymptotique du nombre de géodésiques fermées sur la surface modulaire en courbure non constante. \'Etudes spectrales d'op\'erateurs de transfert et applications.} Ast\'erisque 238, 111-177 (1996).

\bibitem{doney} 
{\sc Doney, R. A.}
{\em An analogue of the renewal theorem in higher dimensions.}
Proc. London Math. Soc. (3) 16 1966 669--684. 

\bibitem{duf}
{\sc Duflo M.}
{\em Random Iterative Models.} 
Applications of Mathematics, Springer-Verlag Berlin Heidelberg (1997). 

\bibitem{fel}
{\sc Feller W.}
{\em An introduction to probability theory and its applications, Vol.~II.}
John Wiley and Sons, New York (1971).

\bibitem{DLJ}
{\sc Ferré D., Hervé L., Ledoux J.} 
{\em Limit theorems for stationary Markov processes with $\L^2$-spectral gap.} 
To appear in Ann.~I.~H.~Poincaré.

\bibitem{fuhlai}
{\sc Fuh C.D, Lai T.L.}
{\em Asymptotic expansions in multidimensional Markov renewal theory and first passage times for Markov random walks.}
Adv. in Appl. Probab. {\bf 33}, 652-673 (2001). 

\bibitem{denis}
{\sc Guibourg D.}
{\em Th\'eor\`eme de renouvellement pour  cha\^{\i}nes de Markov fortement ergodiques. 
Applications aux mod\`eles it\'eratifs Lipschitziens.} 
C. R. Acad. Sci. Paris, Ser. I 346 (2008) 435-438.  

\bibitem{denis-these}
{\sc Guibourg D.}
{\em Théorèmes de renouvellement pour des fonctionnelles additives associées à des chaînes de Markov fortement ergodiques.} Ph.D.~Thesis, INSA-IRMAR Rennes, 2011.  http://hal.archives-ouvertes.fr/docs/00/58/31/75/PDF/Guibourg-These.pdf

\bibitem{guiher}
{\sc Guibourg D., Herv\'e L.}
{\em  A renewal theorem for strongly ergodic Markov chains in dimension $d\geq3$ and in the centered case. } 
Potential Analysis, 34, 385-410 (2011). 

\bibitem{fl}
{\sc Herv\'e L, P\`ene F.}
{\em The Nagaev-Guivarc'h method via the Keller-Liverani theorem. } 
Bull. Soc. Math. France, {\bf 138} (2010) 415-489. 

\bibitem{JLV}  
{\sc Hervé L., Ledoux J., Patilea V.} 
{\em A Berry-Esseen theorem on $M$-estimators for geometrically ergodic Markov chains.} 
To appear in Ann. Inst. H. Poincaré. 

\bibitem{keli}
{\sc Keller G., Liverani C.}
{\em Stability of the Spectrum for Transfer Operators.} 
Ann. Scuola Norm. Sup. Pisa. CI. Sci. (4) Vol. XXVIII (1999) 141-152. 

\bibitem{lalley}
{\sc Lalley S.}
{\em Renewal theorems in symbolic dynamics, with applications to geodesic flows, noneuclidean tessellations and their fractal limits.} 
Acta Math.~{\bf 163} (1989), pp.~1-55. 

\bibitem{mey}  
{\sc S.P. Meyn and R.L. Tweedie.}
{\em Markov chains and stochastic stability.} 
Springer Verlag, New York, Heidelberg, Berlin (1993). 

\bibitem{ney-spit}
{\sc Ney P., Spitzer F.}
{\em The Martin boundary for random walk.}
Trans. Amer. Math. Soc. 121 1966 116--132. 

\bibitem{rosen}
{\sc Rosenblatt M. }
{\em Markov processes. Structure and asymptotic behavior. } 
Springer-Verlag. New York (1971). 

\bibitem{smith}
{\sc Smith W. L.}
{\em A frequency function form of the central limit theorem. } 
Proc.~Cambridge Phil.~Soc., {\bf 49}, 462-472, 1953. 

\bibitem{spitzer}
{\sc Spitzer F. }
{\em Principles of random walks. } 
Van Nostrand, Princeton, 1964. 

\bibitem{stam}
{\sc Stam A. J.}
{\em Renewal theory in $r$ dimensions.}
Compositio Math. 21 1969 383--399. 

\bibitem{thirion}
{\sc Thirion X.}
{\em Propriétés de mélange du flot des chambres de Weyl des groupes de Ping-Pong.} 
Bull.~Soc.~Math.~France, {\bf 137}, 3, pp.~387-421 (2009). 

\bibitem{uchi}
{\sc Uchiyama K. }
{\em Asymptotic estimates of the Green functions and transition probabilities for Markov additive processes. } 
Electronic journal of Probability, {\bf 12}, pp.~138-180 (2007).

\bibitem{wat}
{\sc Watson G. N. }
{\em A treatise on the theory of Bessel functions. }
Cambridge University Press, 1966.  


\end{thebibliography}
\end{document}